\documentclass[oneside]{amsart}
\usepackage{etex}
\usepackage[a4paper]{geometry}
\usepackage{amssymb,amsfonts,amsmath}
\usepackage{comment} 
\usepackage[all,arc]{xy}
\usepackage{enumerate}
\usepackage{mathrsfs,mathtools}
\usepackage{todonotes,booktabs}
\usepackage{stmaryrd}
\usepackage{marvosym}
\usepackage{tikz}
\usepackage{tikz-cd}
\usetikzlibrary{trees}
\usetikzlibrary[shapes]
\usetikzlibrary[arrows]
\usetikzlibrary{patterns}
\usetikzlibrary{fadings}
\usetikzlibrary{backgrounds}
\usetikzlibrary{decorations.pathreplacing}
\usetikzlibrary{decorations.pathmorphing}
\usetikzlibrary{positioning}
\def\biblio{\bibliography{duality}\bibliographystyle{alpha}}
\usepackage{xcolor} %improved colors
\usepackage{graphicx}

\usepackage[pagebackref]{hyperref}

\definecolor{dark-red}{rgb}{0.5,0.15,0.15}
\definecolor{dark-blue}{rgb}{0.15,0.15,0.6}
\definecolor{dark-green}{rgb}{0.15,0.6,0.15}
\hypersetup{
    colorlinks, linkcolor=dark-red,
    citecolor=dark-blue, urlcolor=dark-green
}
\newcommand{\iHom}{\underline{\operatorname{Hom}}}

\renewcommand*{\backref}[1]{}
\renewcommand*{\backrefalt}[4]{%
  \ifcase #1 %
No citations.% use \relax if you do not want the "No citations" message
  \or
(cit. on p. #2).%
  \else
(cit on pp. #2).%
  \fi%
}
\usepackage[nameinlink,capitalise,noabbrev]{cleveref}

\newtheorem{thm}{Theorem}[section]
\newtheorem{cor}[thm]{Corollary}
\newtheorem{prop}[thm]{Proposition}
\newtheorem{lem}[thm]{Lemma}

\newtheorem{quest}[thm]{Question}

\newtheorem*{thm*}{Theorem}
\newtheorem*{cor*}{Corollary}

\theoremstyle{definition}
\newtheorem{defn}[thm]{Definition}

\newtheorem{exmp}[thm]{Example}

\theoremstyle{remark}
\newtheorem{rem}[thm]{Remark}

\makeatletter
\let\c@equation\c@thm
\makeatother
\numberwithin{equation}{section}

\DeclareMathOperator{\Sp}{Sp}
\DeclareMathOperator{\Hom}{Hom}

\DeclareMathOperator{\End}{End}
\DeclareMathOperator{\colim}{colim}
\DeclareMathOperator{\cA}{\mathcal{A}}
\DeclareMathOperator{\cC}{\mathcal{C}}
\DeclareMathOperator{\cD}{\mathcal{D}}

\DeclareMathOperator{\cS}{\mathcal{S}}
\DeclareMathOperator{\cT}{\mathcal{T}}

\DeclareMathOperator{\cF}{\mathcal{F}}
\DeclareMathOperator{\cG}{\mathcal{G}}
\DeclareMathOperator{\cV}{\mathcal{V}}
\DeclareMathOperator{\Id}{\mathrm{Id}}

\DeclareMathOperator{\Ext}{Ext}
\DeclareMathOperator{\Tor}{Tor}
\DeclareMathOperator{\Mod}{Mod}
\DeclareMathOperator{\Stable}{Stable}

\DeclareMathOperator{\Comod}{Comod}
\DeclareMathOperator{\StMod}{StMod}

\DeclareMathOperator{\Kos}{Kos}

\DeclareMathOperator{\Loc}{Loc}
\newcommand{\Locid}[1]{\mathrm{Loc}_{#1}^{\otimes}}

\DeclareMathOperator{\Ho}{Ho}
\DeclareMathOperator{\Thick}{Thick}
\DeclareMathOperator{\Ind}{Ind}

\DeclareMathOperator{\Fun}{Fun}

\DeclareMathOperator{\Alg}{Alg}

\newcommand{\bE}{\mathbb{E}}
\DeclareMathOperator{\cell}{cell}
\DeclareMathOperator{\ac}{ac}
\DeclareMathOperator{\local}{local}
\DeclareMathOperator{\Stablec}{\mathcal{C}}
\DeclareMathOperator{\Cell}{Cell}

\newcommand{\matlis}{\mathbb{D}}

\newcommand{\p}{\frak{p}}

\newcommand{\kos}[2]{{#1}/\!\!/{#2}}

\newcommand{\cal}{\mathcal}
\newcommand{\xr}{\xrightarrow}
\newcommand{\cH}{\check{C}H}

\newcommand{\Z}{\mathbb{Z}}

\Crefname{figure}{Figure}{Figures}
\Crefname{assu}{Assumption}{Assumptions}
\Crefname{lem}{Lemma}{Lemmas}
\Crefname{prop}{Proposition}{Propositions}

\renewcommand{\frak}{\mathfrak}

\newcommand{\vP}{\cal{V}{(\frak{p})}}
\newcommand{\Spec}{\textnormal{Spec}^{\textnormal{h}}}

\newcommand{\cGorenstein}{Absolute Gorenstein }
\newcommand{\Gorenstein}{absolute Gorenstein }
\newcommand{\Normalization}{Gorenstein normalization }

\newcommand{\nsGorenstein}{absolute Gorenstein}

\DeclareMathOperator{\Inj}{Inj}

\newcommand{\fp}{\mathfrak{p}}
\newcommand{\fq}{\mathfrak{q}}
\newcommand{\fr}{\mathfrak{r}}

\newcommand{\recollement}[5]{
\xymatrix{{#1} \ar[r]|-{#2} & #3 \ar[r]|-{#4} \ar@<1ex>[l]^-{{#2}_!} \ar@<-1ex>[l]_-{{#2}^*} & #5, \ar@<1ex>[l]^-{{#4}!} \ar@<-1ex>[l]_-{{#4}^*}
}}
\let\lim\relax

\DeclareMathOperator{\lim}{lim}
\newcommand{\mm}{/\!\!/}

\newcommand{\cU}{\mathcal{U}}

\title{Local duality for structured ring spectra}

\author{Tobias Barthel}
\address{Department of Mathematical Sciences, University of Copenhagen, Universitetsparken 5, 2100 K{\o}benhavn {\O}, Denmark}
\email{tbarthel@math.ku.dk}
\author{Drew Heard}
\address{Universit{\"a}t Hamburg, Bundesstrasse 55, 20146 Hamburg, Germany}
\email{drew.k.heard@gmail.com}
\author{Gabriel Valenzuela}
\address{Department of Mathematics, The Ohio State University, 100 Math Tower, 231 West 18th Avenue, Columbus, OH 43210-1174, USA}
\email{valenzuelavasquez.2@osu.edu}

\date{\today}

\begin{document}

\begin{abstract}
We use the abstract framework constructed in our earlier paper \cite{bhv} to study local duality for Noetherian $\bE_{\infty}$-ring spectra. In particular, we compute the local cohomology of relative dualizing modules for finite morphisms of ring spectra, thereby generalizing the local duality theorem of Benson and Greenlees. We then explain how our results apply to the modular representation theory of compact Lie groups and finite group schemes, which recovers the theory previously developed by Benson, Iyengar, Krause, and Pevtsova.
\end{abstract}

\maketitle

%{\hypersetup{linkcolor=black}\tableofcontents}
%\tableofcontents
\def\biblio{}

\section{Introduction}\label{sec:introduction}

\subsection{Background and motivation}

In \cite{bhv}, we developed an abstract framework for constructing local cohomology and local homology functors for a general class of stable $\infty$-categories, and used it to demonstrate the ubiquity of local duality in algebra and topology. The goal of this paper is to investigate a specific class of examples in detail, applying our methods in particular to modular representation theory. Moreover, we consider the relation between local and global duality for structured ring spectra, thereby providing a different perspective on the Gorenstein condition previously studied in depth by Greenlees~\cite{greenlees_hi}. The main examples of $\infty$-categories of interest to us in this paper are coming from the modular representation theory of a finite group $G$ over a field $k$ of characteristic $p$ dividing the order of $G$. There are two natural stable $\infty$-categories associated to the group algebra $kG$: The derived category $\cD_{kG} = \cD\Mod_{kG}$ and the stable module category $\StMod_{kG}$. Benson and Krause \cite{krausebenson_kg} constructed a single triangulated category $\Stable_{kG}$ fitting into a recollement
\[
\xymatrix{\StMod_{kG} \ar[r] & \Stable_{kG} \ar[r] \ar@<1ex>[l] \ar@<-1ex>[l] & \cD_{kG},  \ar@<1ex>[l] \ar@<-1ex>[l]}
\]
hence containing the information of both of the outer terms. If $G$ is a $p$-group, Morita theory provides an equivalence between $\Stable_{kG}$ and the homotopy category of module spectra over the $\bE_{\infty}$-ring spectrum $C^*(BG,k)$ of cochains on $G$ with coefficients in $k$, and it is straightforward to lift this to an equivalence of underlying $\infty$-categories. However, for general finite $G$, $\Stable_{kG}$ is not generated by $k$ and hence $\Mod_{C^*(BG,k)}$ forms a proper localizing subcategory. 

Benson, Iyengar, and Krause \cite{benson_local_cohom_2008} developed a theory of local cohomology and homology functors for certain triangulated categories, which was employed in remarkable work \cite{bik_annals} to classify all localizing subcategories of $\Stable_{kG}$. Our first aim is to show that their theory is equivalent to our abstract local duality framework when applied to the cellular objects in $\Stable_{kG}$. This requires a mild modification and generalization of our techniques, which we develop in \Cref{sec:abstract}. Specifically, we formulate abstract local duality for torsion subcategories that are not necessarily ideals. This allows us to study cellular objects in rather general stable $\infty$-categories, thus covering examples like the motivic homotopy category as special cases. 

Our approach can then be described in two steps:
\begin{enumerate}
	\item Formulate and prove local duality for Noetherian $\bE_{\infty}$-ring spectra $R$, which applies in particular to $R = C^*(BG,k)$.
	\item Transfer this theory to the subcategory of cellular objects in $\Stable_{kG}$ via derived Morita theory.
	%\item Extend the corresponding local cohomology and local homology functors from cellular objects to the whole $\infty$-category $\Stable_{kG}$ via an abstract extension theory. 
\end{enumerate}

While this seems to be more complicated than the approach taken in \cite{benson_local_cohom_2008} at first sight, it has several advantages. First of all, $\infty$-categories of module spectra are compactly generated by the free module of rank 1, and are therefore amenable to the methods of \cite{bhv}. Moreover, since we can directly work with the coefficients of $R$ rather than an action as in \cite{benson_local_cohom_2008}, the construction of several spectral sequences computing local cohomology and local homology becomes more transparent, allowing us to verify convergence in many new cases. Finally, as we will see shortly, the theory for Noetherian $\bE_{\infty}$-ring spectra is interesting in its own right, as it provides a toy example for potential generalizations to derived algebraic geometry. 

\subsection{Main results}

We proceed to briefly summarize the main theorems of the present paper. Building on the abstract framework developed in \Cref{sec:abstract}, the main result of \Cref{sec:ringspectra} is the construction of local cohomology and local homology functors for structured ring spectra in terms of specialization closed subsets of the Zariski spectrum of $\pi_*R$. Informally speaking, this theory provides a notion of support for affine derived schemes analogous to \cite{benson_local_cohom_2008,bikp_stratification}. 

\begin{thm*}[\Cref{thm:localduality}]
Let $R$ be an $\bE_{\infty}$-ring spectrum with $\pi_*R$ a graded Noetherian ring. For any specialization closed subset $\cV \subseteq \Spec(\pi_*R)$ there exists a quadruple of functors $(\Gamma_{\cV},L_{\cV},\Delta_{\cV},\Lambda_{\cV})$ on $\Mod_R$ satisfying the local duality properties listed in \emph{loc.~cit.}. Moreover, these functors coincide with the local cohomology and local homology functors previously constructed by Benson, Iyengar, and Krause \cite{benson_local_cohom_2008}. 
\end{thm*}

Following the approach outlined above, we obtain the local cohomology and local homology functors we were after, not just for finite groups, but in fact for finite group schemes. After introducing an appropriate $\infty$-category of comodules over Hopf algebras, the following result is proved in \Cref{sec:hopfalgebras} in terms of finite-dimensional Hopf algebras; for variety, we state it here in its equivalent geometric form; see for example \cite[Ch.~II]{gr_book} for more details on ind-coherent sheaves. 

\begin{thm*}[\Cref{thm:cellular_functors}]
If $\mathbb{G}$ denotes a finite affine group scheme over a field $k$, then the $\infty$-category $\mathrm{IndCoh}^{\cell}_{\mathbb{G}}$ of cellular ind-coherent sheaves on $\mathbb{G}$ admits local cohomology $\Gamma_{\cV}$ and local homology functors $\Lambda_{\cV}$ satisfying local duality for every specialization closed subset $\cV \subseteq \Spec(H^*(\mathbb{G},k))$. 
\end{thm*}

This theorem comes with a strongly convergent spectral sequence computing the local cohomology of a cellular ind-coherent sheaf $\cF$ from the algebraic local cohomology of its cohomology, see \Cref{prop:localcohomssnew},
\[
E^2_{s,t}\cong(H^{s}_{\frak p}H^*(\mathbb{G},\cF)_{\frak p})_{t} \implies H^{s+t}(\mathbb{G},\Gamma_{\frak p}\cF),
\]
generalizing a similar spectral sequence for finite groups and the trivial module constructed earlier by Greenlees and Lyubeznik \cite{green_lyub}.

As one application of our framework, we formulate and prove an analog of the chromatic splitting conjecture for Noetherian $\mathbb{E}_{\infty}$-ring spectra in the companion paper \cite{bhv3}. In the present paper, we instead turn to the study of the relative dualizing module and Gorenstein duality for $\bE_{\infty}$-ring spectra in \Cref{sec:gorenstein}. Generalizing the notion of Gorenstein for discrete commutative rings, we call a structured ring spectrum $R$ absolute Gorenstein whenever it satisfies Gorenstein duality for all homogeneous prime ideals $\fp$ in $\pi_*R$, which means that the local cohomology $\pi_*\Gamma_{\fp}R$ of $R$ at $\fp$ is given by the injective hull of the residue field $(\pi_*R)/\fp$. This notion is closely related to Greenlees' definition of Gorenstein duality, see \cite{dgi_duality,greenlees_hi}, but does not make reference to a residue ring spectrum $k$ lifting $(\pi_*R)/\fp$. 

Furthermore, recall that the forgetful functor $f^*\colon \Mod_S \to \Mod_R$ corresponding to a map of ring spectra $f\colon R \to S$ has both a left and a right adjoint, given by induction $f_*$ and coinduction $f_!$, respectively. The coinduced module $\omega_f = f_!(R) \in \Mod_S$ is the relative dualizing module; for example, if $f^*$ preserves compact objects, then it satisfies Grothendieck duality:
\[
f_*(-) \otimes \omega_f \simeq f_!(-),
\]
see \cite{bal_gn_dual}. Our next theorem computes the local cohomology of the relative dualizing module under certain conditions, thereby relating global and local duality. 

\begin{thm*}[\Cref{thm:bc}]
Suppose $f\colon R \to S$ is a morphism of Noetherian $\bE_{\infty}$-ring spectra with $R$ Gorenstein, such that $S$ is a compact $R$-module, then for every homogeneous prime ideal $\fp$ in $\pi_*S$ there is an isomorphism 
\[
\pi_*(\Gamma_{\fp}\omega_f) \cong I_{\fp}
\]
of $\pi_*S$-modules, where $I_{\fp}$ is the injective hull of the residue field $(\pi_*S)/{\fp}$.
\end{thm*}

This result has two consequences. Firstly, the given formula lifts to an equivalence $(\Gamma_{\fp}S) \otimes_S \omega_f \simeq \mathbb{I}_{\fp}$ of $S$-modules, where $\mathbb{I}_{\fp}$ denotes the Brown--Comenetz dual of $S$ associated to $\fp$. This equivalence can be read as a version of Gorenstein duality for $S$ twisted by the relative dualizing module $\omega_f$. In other words, under the assumptions of the theorem, Gorenstein duality for $R$ ascends to $S$ up to a twist. Inspired by Noether normalization in commutative algebra, for a given ring spectrum $S$ one can often find a morphism $f$ satisfying the conditions of the theorem and such that $\omega_f$ is an invertible $S$-module. When this is the case, we refer to $f$ as a Gorenstein normalization of $S$.

\begin{cor*}%[\Cref{cor:bc}]
If $S$ admits Gorenstein normalization, then $S$ has Gorenstein duality with twist $\omega_f$ for all prime ideals $\fp$ in $\pi_*S$, i.e., it satisfies $(\Gamma_{\fp}S) \otimes_S \omega_f \simeq \mathbb{I}_{\fp}$.
\end{cor*}

We then exhibit several examples of ring spectra which admit (twisted) Gorenstein duality, as for instance the $\bE_{\infty}$-algebra of cochains $C^*(BG,k)$ on a compact Lie group $G$ with coefficients in a field $k$. This recovers and generalizes Benson's conjecture as proven by Benson and Greenlees~\cite{bg_localduality} for compact Lie groups satisfying a certain orientability condition. 

Secondly, we may interpret the above theorem as a shadow of a residual complex formalism in derived algebraic geometry. To see this, recall that Grothendieck and Hartshorne \cite{hartshorne_resduality} construct dualizing complexes by gluing together local dualizing complexes according to a natural stratification of a given scheme. Heuristically, the local cohomology at $\fp$ of the resulting complex should then isolate the local dualizing complex at $\fp$, which is precisely the content of our theorem above. This observation motivates to seek a construction of residual complexes for spectral schemes; we hope to return to this point in a future paper.

\subsection{Notation and conventions}

Throughout this paper, we will work in the setting of $\infty$-categories as developed in \cite{htt,ha}, and will use the local duality framework described in \cite{bhv}. In particular, all constructions will implicitly be assumed to be derived, unless otherwise stated. A \emph{stable category} is a symmetric monoidal stable $\infty$-category $\cC=(\cC,\otimes)$ which is compactly generated by dualizable objects and whose monoidal product $\otimes$ commutes with colimits separately in each variable. Writing $A$ for a unit of the stable category $(\cC,\otimes)$, we define the (Spanier--Whitehead) dual of an object $X \in \cC$ by $X^{\vee} = \iHom_{\cC}(X,A)$, where $\iHom_{\cC}$ denotes the internal mapping object of $\cC$; note that, under our assumptions on $\cC$, $\iHom_{\cC}$ exists for formal reasons. 

A subcategory of $\cT \subseteq \cC$ is called thick if it is closed under finite colimits, retracts, and desuspensions, and $\cT$ is called localizing if it is closed under all filtered colimits as well. Furthermore, a thick subcategory $\cT$ in $\cC$ is a thick ideal if $T \otimes X \in \cT$ for all $T \in \cT$ and all $X \in \cC$, and the notion of localizing ideal is defined analogously. For a collection of objects $\cS \subseteq \cC$, we denote the smallest localizing subcategory of $\cC$ containing $\cS$ by $\Loc_{\cC}(\cS)$, and $\Loc^{\otimes}_{\cC}(\cS)$ for the localizing ideal of $\cC$ generated by $\cS$. If the ambient category is clear from context, we also omit the subscript $\cC$. 

All discrete rings $R$ in this paper are assumed to be commutative and graded, and all ring-theoretic notions are implicitly graded. In particular, an $R$-module $M$ refers to a graded $R$-module and we write $\Mod_R$ for the abelian category of discrete graded $R$-modules. Prime ideals in $R$ will be denoted by fraktur letters $\fp,\fq,\fr$ and are always homogeneous, so that $\Spec(R)$ refers to the Zariski spectrum of homogeneous prime ideals in $R$. 

Finally, our grading conventions are homological; differentials always decrease degree. Thus, for example, $\pi_{-i}C^*(BG,k) \cong H^{i}(G,k)$, i.e., $C^*(BG,k)$ is a coconnective $\bE_{\infty}$-ring spectrum.

\subsection*{Acknowledgements}

We would like to thank John Greenlees, Srikanth Iyengar, Henning Krause, and Vesna Stojanoska for helpful discussions, and the referee for many useful comments and corrections. Moreover, we are grateful to the Max Planck Institute for Mathematics for its hospitality, funding two week-long visits of the third-named author.

\section{Local duality contexts and recollements}\label{sec:abstract}

The goal of this section is to set up an abstract framework in which local duality for stable categories can be conveniently studied. The material of the first two subsections is well known to the experts, and can be skipped on a first reading. 

\subsection{Abstract local duality}

We recall some definitions, constructions, and results from \cite{bhv}, in a slightly more general setup which allows localizing subcategories that are not necessarily tensor ideals. This extra generality is required for the applications to cellular objects in later sections. 

\begin{defn}
A symmetric monoidal stable $\infty$-category $\cC = (\cC,\otimes,A)$ with symmetric monoidal structure $\otimes$ and unit $A$ is called a stable category if it satisfies the following conditions:
\begin{enumerate}
	\item The $\infty$-category $\cC$ is compactly generated by a set $\cG$ of (strongly) dualizable objects in $\cC$. In particular, $\cC$ is presentable.  
	\item The symmetric monoidal product $\otimes$ preserves colimits separately in each variable.
\end{enumerate}
\end{defn}

We refer the reader to \cite[Sec.~2]{bhv} for a summary of the basic properties of stable categories and their localizing subcategories. 

For a localizing subcategory $\cT \subseteq \cC$, we often write $\iota_{\cT}\colon \cT \to \cC$ for the corresponding inclusion functor; when no confusion is likely to occur, $\iota_{\cT}$ is applied implicitly whenever necessary.  
Moreover, we denote by $\cT^{\perp}$ the left-orthogonal of $\cT$ in $\cC$, i.e., the thick subcategory of $\cC$ generated by all objects $X$ for which $\Hom(\cT,X)\simeq 0$. 

\begin{defn}
A local duality context $(\cC,\cT)$ consists of
\begin{enumerate}
	\item a stable category $\cC$, and 
	\item a localizing subcategory $\cT \subseteq \cC$ which is generated by compact objects in $\cC$. 
\end{enumerate}
If $\cC$ is monogenic and $\cT$ is generated by a single compact object $T \in \cC$, then we sometimes write $(\cC,T)$ for $(\cC,\Loc(T))$. 
\end{defn}
As an example for the reader to keep in mind, one can consider $(\Mod_R,R/I)$ for $R$ a Noetherian commutative ring and $I \subseteq R$ an ideal generated by a regular sequence, see \cite[Sec.~3]{bhv}.
\begin{rem}
This definition is a mild generalization of the one given in \cite{bhv}, in which we additionally required $\cT$ to be a localizing ideal. The main reason for extending our definition is the application to categories of cellular objects in \Cref{ssec:cell}. Since the proofs work equally well for localizing subcategories, we will not reproduce the arguments here; see also the discussion in \cite[Rem.~3.1.3]{hps_axiomatic}.
\end{rem}

A local duality context generates a number of categories and functors with many good properties, which are summarized in the next theorem. 

\begin{thm}\label{thm:abstractlocalduality}
Let $(\cC,\otimes,A)$ be a stable category compactly generated by dualizable objects. If $(\cC,\cT)$ is a local duality context, then $\cT^{\perp} \subseteq \cC$ is a localizing subcategory and the following properties hold:
\begin{enumerate}
	\item The inclusion $\iota_{\cT}$ has a right adjoint $\Gamma = \Gamma_{\cT}$, and the inclusions $\iota_{\cT^{\perp}}$ and $\iota_{\cT^{\perp\perp}}$ have left adjoints $L = L_{\cT}$ and $\Lambda = \Lambda_{\cT}$, respectively, which induce natural cofiber sequences
 \[
 \Gamma X \longrightarrow X \longrightarrow LX  \quad \text{and} \quad \Delta X \longrightarrow X \longrightarrow \Lambda X
 \]
for all $X \in \cC$. Here, $\Delta$ is right adjoint to $\iota_{\cT^{\perp}}$. In particular, $\Gamma$ is a colocalization functor and both $L$ and $\Lambda$ are localization functors. 
	\item Both $\Gamma \colon \cC \to \cC$ and $L\colon \cC \to \cC$ preserve all colimits. Moreover, $L$ preserves compact objects. 
	\item The functors $\Lambda'=\Lambda \iota_{\cT}\colon \cT \to \cT^{\perp\perp}$ and $\Gamma'=\Gamma \iota_{\cT^{\perp\perp}}\colon \cT^{\perp\perp} \to \cT$ are mutually inverse equivalences of stable categories. Moreover, there are natural equivalences of functors
\[
\xymatrix{\Lambda \Gamma \ar[r]^-{\sim} & \Lambda & \mathrm{and} &\Gamma \ar[r]^-{\sim} & \Gamma \Lambda.}
 \]
	\item The functors $(\Gamma,\Lambda)$ form an adjoint pair, so that we have a natural equivalence 
\[
\xymatrix{\Hom(\Gamma X,Y) \simeq \Hom(X,\Lambda Y)}
\]
for all $X,Y \in \cC$. Similarly, viewed as endofunctors on $\cC$, $L$ is left adjoint to $\Delta$. 
	\item There is natural equivalence $\Lambda L \simeq \Sigma\Delta\Gamma$.
	\item There is a homotopy pullback square of functors 
	\[
	\xymatrix{\Id \ar[r] \ar[d] & \Lambda \ar[d] \\
	L \ar[r] & L\Lambda,}
	\]
	which is usually referred to as the fracture square. 
\end{enumerate}
\end{thm}
\begin{proof}
The first four parts are given in \cite[Thm.~2.21]{bhv} as an $\infty$-categorical version of \cite[Thm.~3.3.5]{hps_axiomatic}, while Parts (5) and (6) are contained in \cite[Cor.~2.26, Rem.~2.29]{bhv} extending work of Greenlees~\cite{greenlees_axiomatic}. 
\end{proof}

The categories and functors considered in this theorem can be organized in the following diagram of adjoints,
\begin{equation}\label{eq:abstractdiagram}
\begin{gathered}
\xymatrix{& \cT^{\perp} \ar[d] \ar@{-->}@/^1.4pc/[ddr] \\
& \cC \ar@<1ex>[u]^{L} \ar@<-1ex>[u]_{\Delta} \ar@<0.5ex>[ld]^{\Gamma} \ar@<0.5ex>[rd]^{\Lambda} \\
\cT \ar@<0.5ex>[ru] \ar[rr]_{\sim} \ar@{-->}@/^1.4pc/[ruu] & & \cT^{\perp\perp}, \ar@<0.5ex>[lu]}
\end{gathered}
\end{equation}
where the dotted arrows indicate left orthogonality.

\begin{rem}
We refer to the equivalence of \Cref{thm:abstractlocalduality}(4) as abstract local duality. This choice of terminology is justified by specializing to classical local duality contexts, as for example $(\Mod_R,R/I)$ for $R$ a suitable commutative ring and $I \subseteq R$ an ideal, see \cite[Sec.~3]{bhv}. In the case where $R = \Z$ so that $\Mod_{\Z}$ refers to the (unbounded) derived $\infty$-category of abelian groups, taking $I = (p)$, we get the usual notations of $p$-torsion, $p$-local, and $p$-complete $\Z$-modules, studied in, for example, \cite[Sec.~3.1]{dwyer_complete_2002}. 
\end{rem}

The only difference when working with a localizing category $\cT$ which is not closed under tensoring with objects in $\cC$ occurs in Part (3) of \Cref{thm:abstractlocalduality}, where the cocontinuity of $\Gamma$ and $L$ does not imply that they are smashing. For completeness, we state the case of a localizing ideal separately.

\begin{lem}\label{lem:smashing}
Suppose $(\cC,\cT)$ is a local duality context and assume that $\cT$ is a localizing ideal, i.e., $\cT$ is closed under tensoring with arbitrary objects in $\cC$. The left orthogonal $\cT^{\perp}$ is then a localizing ideal as well, and both $\Gamma$ and $L$ are smashing functors, i.e., $\Gamma(X) \simeq X \otimes \Gamma(A)$ for all $X \in \cC$ and similarly for $L$. 
\end{lem}
\begin{proof}
This is proven as in \cite[Thm.~3.3.5]{hps_axiomatic}; for the $\infty$-categorical version of this result, see \cite[Thm.~2.21]{bhv}.
\end{proof}

\subsection{Recollements}

In order to facilitate the comparison between our construction of local cohomology and local homology functors to other approaches, we now explain how to extract two equivalent recollements from a local duality context. These correspond to choosing either of the equivalent subcategories $\cT$ or $\cT^{\perp\perp}$. For a more detailed discussion of recollements in the setting of $\infty$-categories, see \cite[A.8]{ha}.

\begin{defn}\label{defn:recollement}
A sequence of functors of stable $\infty$-categories
\[
\xymatrix{\cC_0 \ar[r]^-g & \cC \ar[r]^-f & \cC_1}
\]
is a recollement if it is both a localization and a colocalization sequence. More explicitly, this means that $f$ and $g$ admit left and right adjoints $(f_* \dashv f \dashv f_!)$ and $(g_* \dashv g \dashv g_!)$, respectively, forming a colocalization and a localization sequence
\[
\xymatrix{\cC_0 \ar[r]^-g & \cC \ar[r]^-f \ar@<1ex>[l]^-{g_*} & \cC_1 \ar@<1ex>[l]^-{f_*} & \text{and} & \cC_0 \ar[r]_-g & \cC \ar[r]_-f\ar@<-1ex>[l]_-{g_!} & \cC_1  \ar@<-1ex>[l]_-{f_!}.}
\]
For more details and localization and colocalization sequences (at least on the level of triangulated categories) we refer the reader to \cite[Sec.~3]{krausestablederivedcat}. 

A recollement can be displayed more compactly as:
\begin{equation}\label{eq:recollementlabels}
\xymatrix{\cC_0 \ar[r]|-g & \cC \ar[r]|-f \ar@<1ex>[l]^-{g_*} \ar@<-1ex>[l]_-{g_!} & \cC_1 \ar@<1ex>[l]^-{f_*} \ar@<-1ex>[l]_-{f_!}.}
\end{equation}
\end{defn} 

\begin{rem}
The notions of localization sequence and colocalization sequence are dual to each other. In order to avoid confusion, our convention is to always read them from left to right.
\end{rem}

The next result collects a number of basic properties of recollements that we will use frequently throughout this paper. 

\begin{lem}\label{lem:recollementproperties}
Suppose $\cC_0 \xrightarrow{g} \cC \xrightarrow{f} \cC_1$ is a recollement of compactly generated stable $\infty$-categories with notation as in \eqref{eq:recollementlabels}, then the following properties are satisfied:
\begin{enumerate}
	\item The functors $f_*$, $g$, and $f_!$ are fully faithful.  
	\item The functors $f_*$ and $g_*$ preserve compact objects. 
	\item In the colocalization sequence, $g$ exhibits $\cC_0$ as the kernel of  $f$, and $g_*$ exhibits $\cC_0$ as the Verdier quotient $\cC/\cC_1$. Dually, $f_!$ exhibits $\cC_1$ as the kernel of $g_!$, and $f$ exhibits $\cC_1$ as the quotient $\cC/\cC_0$.
\end{enumerate}
\end{lem}
\begin{proof}
Statements (1) and (3) are part of the definition of a localization sequence. For (2), first note that $f$ and $g$ preserve all colimits. By \cite[Prop.~5.5.7.2]{htt}, this is equivalent to the fact that their left adjoints $f_*$ and $g_*$ preserve compact objects. 
\end{proof}

\begin{prop}\label{prop:abstractrecollement}
If $(\cC,\cT)$ is a local duality context then, using the notation introduced in \Cref{thm:abstractlocalduality}, there exist two recollements
\begin{equation}\label{eq:recollementsfromldc}
\xymatrixcolsep{3pc}
\xymatrix{\cT^{\perp} \ar[r]|-{\iota_{\cT^{\perp}}} & \cC \ar[r]|-{\Gamma} \ar@<1.2ex>[l]^-{L} \ar@<-1.2ex>[l]_-{\Delta} & \cT \ar@<1.2ex>[l]^-{\iota_{\cT}} \ar@<-1.2ex>[l]_-{\iota_{\cT^{\perp\perp}}\Lambda'} & \text{and} & \cT^{\perp} \ar[r]|-{\iota_{\cT^{\perp}}} & \cC \ar[r]|-{\Lambda} \ar@<1.2ex>[l]^-{L} \ar@<-1.2ex>[l]_-{\Delta} & \cT^{\perp\perp} \ar@<1.2ex>[l]^-{\iota_{\cT}\Gamma'} \ar@<-1.2ex>[l]_-{\iota_{\cT^{\perp\perp}}}.}
\end{equation}
Furthermore, the mutual equivalences of \Cref{thm:abstractlocalduality}(3) induce a natural equivalence between these two recollements. 
\end{prop}
\begin{proof}
This follows from \Cref{thm:abstractlocalduality}: By construction, there are localization sequences
\begin{equation}\label{eq:twolocalseqrecollement}
\xymatrix{\cT \ar[r]^-{\iota_{\cT}} & \cC \ar[r]^-{L} & \cT^{\perp} & \text{and} & \cT^{\perp} \ar[r]^-{\iota_{\cT^{\perp}}} & \cC \ar[r]^-{\Lambda} & \cT^{\perp\perp}.}
\end{equation}
The localization sequence on the right can be modified by the equivalence of \Cref{thm:abstractlocalduality}(3), which provides a commutative diagram
\begin{equation}\label{eq:twist}
\xymatrix{\cC \ar[r]_{\Lambda} \ar[rd]_{\Lambda} & \cT^{\perp\perp} \ar@<-1ex>[l]_{\iota_{\cT^{\perp\perp}}} \ar@<-0.5ex>[d]_{\Gamma'} \\
& \cT \ar@<-0.5ex>[u]_{\Lambda'},}
\end{equation}
where the vertical functors $\Gamma'$ and $\Lambda'$ are mutual equivalences. The dual of the first localization sequence of \eqref{eq:twolocalseqrecollement} is a colocalization sequence, so we obtain the first recollement displayed in \eqref{eq:recollementsfromldc}. Modifying the left localization sequence of \eqref{eq:twolocalseqrecollement} by the equivalence instead results in the other recollement. Finally, the natural equivalence between these two recollements is induced from the equivalences of \eqref{eq:twist}.
\end{proof}

We will now exhibit a method for generating recollements starting from a category $\cC$. As a special instance of this construction, we recover Krause's stable derived category~\cite{krausestablederivedcat}.

\begin{prop}\label{prop:recollements}
Suppose $\cC_1$ is a stable category and $\cD \subseteq \cC_1$ is a small idempotent complete thick subcategory of $\cC_1$ closed under the monoidal product of $\cC_1$. If $\cD$ contains the subcategory $\cC_1^{\omega}$ of compact objects of $\cC_1$, as well as the tensor unit of $\cC_1$, then the pair $(\Ind(\cD),\cC_1)$ is a local duality context. In particular, there exists a recollement
\[
\xymatrix{\cC_0 \ar[r] & \Ind(\cD) \ar[r] \ar@<1ex>[l] \ar@<-1ex>[l] & \cC_1 \ar@<1ex>[l] \ar@<-1ex>[l]}
\]
of stable categories. 
\end{prop}
\begin{proof}
First note that $\Ind(\cD)$ is a stable category, see \cite[Thm.~1.13]{bhv}, and it follows from \cite[Lem.~1.17]{bhv} that $\cC_0$ is as well. By assumption, there exists a commutative diagram
\[
\xymatrix{\cC_1^{\omega} \ar[r]^-j \ar[d]_f & \cC_1 \ar[d]^{\Ind(f)} \\
\cD \ar[r]_-j & \Ind(\cD)}
\]
in which all functors are fully faithful and $\Ind(f)$ preserves all colimits. This square exhibits $\cC_1$ as a localizing subcategory of $\Ind(\cD)$ generated by the image of $\cC_1^{\omega}$ under the composite $jf$. Since $\cC_1^{\omega} \subseteq \cD = \Ind(\cD)^{\omega}$, \Cref{thm:abstractlocalduality} and \Cref{prop:abstractrecollement} apply to yield the claim. 
\end{proof}

\begin{exmp}
Suppose $\cA$ is a symmetric monoidal locally Noetherian Grothendieck abelian category (see \cite{popescu}) and write $\mathrm{noeth}(\cA) \subseteq \cA$ for the full subcategory of Noetherian objects in $\cA$. Assume that the derived category $\cD(\cA)$ is compactly generated, which is the case for example if $\cA$ is generated by a set of objects that remain  compact in $\cD(\cA)$. Since $\cD(\cA)^{\omega} \subseteq \cD^b(\mathrm{noeth}(\cA))$, \Cref{prop:recollements} applies to the pair $(\cD, \cC_1) = (\cD^b(\mathrm{noeth}(\cA)),\cD(\cA))$, giving the recollement studied by Krause in \cite{krausestablederivedcat}.
\end{exmp}

Following Greenlees~\cite{greenlees_axiomatic}, we can construct an abstract version of Tate cohomology in the generality of local duality contexts. 

\begin{defn}
The Tate construction associated to the local duality context $(\cC,\cT)$ is defined as $t_{\cT} = L\Lambda$, viewed as an endofunctor of $\cC$.  
\end{defn}

The next result immediately follows from \Cref{thm:abstractlocalduality}(5); it was originally proven by Greenlees~\cite[Cor.~2.5]{greenlees_axiomatic}, see also~\cite[Rem.~2.29]{bhv}.

\begin{cor}
Suppose $(\cC,\cT)$ is a local duality context. The restriction of the Tate construction $t_{\cT}$ to $\cT$ via $\iota_{\iota_{\cT}}$ is naturally equivalent to $\Sigma\Delta$. 
\end{cor}

\begin{rem}
In fact, Lurie~\cite[A.8.17]{ha} or \cite[Cor.~2.28]{bhv} shows that the data of a recollement \eqref{eq:recollementlabels} is equivalent to giving an exact functor $\cC_1 \to \cC_0$. Unwinding his proof, one can see that the functor corresponding to a recollement arising from a local duality context $(\cC,\cT)$ is precisely the Tate construction $t_{\cT} = L\iota_{\cT^{\perp\perp}}\colon \cT^{\perp\perp} \to \cT^{\perp}$. This together with \Cref{thm:abstractlocalduality}(6) makes precise the sense in which the Tate construction controls the gluing in a recollement. 
\end{rem}

\subsection{Cellularity in stable categories}\label{ssec:cell}

Throughout this section, assume that $\cC = (\cC,\otimes,A)$ is a stable category. By definition, $\cC$ is compactly generated, but it is not necessarily monogenic. Even in the case where $\cC$ is compactly generated by a finite set of objects $\cS \subset \cC^{\omega}$, so that we can form the single compact generator $\bigoplus_{S \in \cS}S$, $\cC$ does not have to be generated by its tensor unit $A$. Consequently, there is no construction of homotopy groups for $\cC$ with all the good properties familiar from the category $\Sp$ of spectra. 

In order to capture the part of $\cC$ that admits a good notion of homotopy groups, we will now introduce and study cellular objects in stable categories. Special cases of these ideas have previously appeared in various contexts, for example in motivic homotopy theory~\cite{duggerisaksencell,totaro_cell} or work in progress by Casacuberta and White, in modular representation theory \cite{krausebenson_kg}, and in equivariant $KK$-theory \cite{kk_theory}. 

\begin{defn}
Suppose $\cC= (\cC,\otimes,A)$ is a stable category. For any $X \in \cC$, we define its homotopy groups as $\pi_*^{\cC}X = \pi_*^{\Sp}\Hom_{\cC}(A,X)$, where the right-hand side are the homotopy groups of the mapping spectrum from $A$ to $X$. We will usually omit the superscript $\cC$ on the homotopy groups if no confusion is likely to occur. 
\end{defn}

\begin{lem}\label{lem:cellulartrivial}
An object $X \in \Loc(A)$ is equivalent to $0$ if and only if $\pi_*^{\cC}X=0$.
\end{lem}
\begin{proof}
Since $A$ is a generator of $\cC^{\cell} = \Loc(A)$, $\pi_*^{\cC}X = \pi_*\Hom_{\cC}(A,X)$ being trivial is equivalent to $X$ being equivalent to $0$. 
\end{proof}

\begin{defn}
The category $\cC^{\cell}$ of cellular objects in $\cC$ is defined as the localizing subcategory of $\cC$ generated by $A$. The right adjoint of the inclusion functor is denoted by
\[
\xymatrix{\Cell\colon \cC \ar[r] &\cC^{\cell} = \Loc_{\cC}(A)}
\]
and will be referred to as cellularization. 
\end{defn}

It follows that there is a colocalization sequence
\[
\xymatrix{\cC^{\cell,\perp} = \cC^{\ac} \ar@<0.5ex>[r] & \cC \ar@<0.5ex>[r] \ar@<0.5ex>[l] & \cC^{\cell} \ar@<0.5ex>[l]}
\]
and we refer to $\cC^{\ac}$ is the category of acyclics in $\cC$. If we assume further that the unit of $\cC$ is compact, we obtain a local duality context in which we can give an explicit description of the associated local homology functor. 

Let $\cC = (\cC,\otimes,A)$ be a stable category with $A \in \cC^{\omega}$ and consider the local duality context $(\cC,\Loc(A))$. The resulting local cohomology and local homology functors are displayed in the following diagram
\begin{equation}\label{eq:cellulardiagram}
\begin{gathered}
\xymatrix{& \cC^{\ac} \ar[d] \ar@{-->}@/^1.4pc/[ddr] \\
& \cC \ar@<-1ex>[u]_{C_{\pi}}\ar@<1ex>[u] \ar@<0.5ex>[ld]^{\Cell} \ar@<0.5ex>[rd]^{L_{\pi}} \\
\cC^{\cell} \ar@<0.5ex>[ru] \ar[rr]_{\sim} \ar@{-->}@/^1.4pc/[ruu] & & \cC^{\pi-\local}. \ar@<0.5ex>[lu]}
\end{gathered}
\end{equation}

The colocalizing subcategory $\cC^{\pi-\local}$ of $\cC$ consists by definition of the $\pi$-local objects in $\cC$. The next result justifies this choice of terminology. 

\begin{lem}\label{lem:bousfield}
The functor $L_{\pi}\colon \cC \to \cC^{\pi-\local}$ is left Bousfield localization with respect to the $\pi_*^{\cC}$-equivalences. In other words, it is the localization functor that inverts the class of those morphisms $f$ in $\cC$ such that $\pi_*^{\cC}f$ is an isomorphism. 
\end{lem}
\begin{proof}
By definition, $L_{\pi}$ is left Bousfield localization away from those objects in $X \in \cC$ such that $\Cell(X) \simeq 0$, i.e., the $\pi$-acyclic objects. By \Cref{lem:cellulartrivial} and adjunction, this condition on $X$ is equivalent to
\[
0 = \pi_*^{\cC}\Cell(X) \cong \pi_*\Hom_{\cC}(A,\Cell(X)) \cong \pi_*\Hom_{\cC}(A,X) \cong \pi_*^{\cC}X.
\]
Therefore, $\pi$-acyclicity of $X$ is equivalent to $\pi_*^{\cC}X$ being trivial, and the claim follows. 
\end{proof}

We now show that the subcategory of cellular objects in $\cC$ is the part of $\cC$ to which the intuition and methods from stable homotopy theory apply most directly. 

\begin{prop}\label{prop:moritacell}
If $\cC = (\cC,\otimes,A)$ is stable category with $A$ compact, then the functor $\Hom_{\cC}(A,-)$ induces an exact and symmetric monoidal equivalence
\[
\xymatrix{\cC^{\cell} \ar[r]^-{\sim} & \Mod_{R}}
\]
with $R = \End_{\cC}(A) \in \Alg_{\bE_{\infty}}(\Sp)$. The inverse equivalence is given by the functor $-\otimes_R A$, using the canonical $R$-module structure on $A$. 
\end{prop}
\begin{proof}
This is an instance of derived Morita theory, as proven originally by Schwede and Shipley~\cite[Thm.~3.1.1]{schwedeshipley_modules} and Lurie \cite[Thm.~7.1.2.1]{ha} in the setting of $\infty$-categories. Since $\cC$ is symmetric monoidal, the unit $A$ is a commutative algebra object in $\cC$, hence $R$ is an $\bE_{\infty}$-ring spectrum. 
\end{proof}

As an application, we deduce the existence of the K{\"u}nneth spectral sequence for cellular objects. Note that, for arbitrary $X,Y \in \cC$, there can be no convergent spectral sequence computing the homotopy groups of $X \otimes Y$ from $\pi_*X$ and $\pi_*Y$.

\begin{cor}
For any $X,Y \in \cC^{\cell}$ there exists a natural and convergent spectral sequence
\[
E_{p,q}^2 \cong \Tor_{p,q}^{\pi_*^{\cC}A}(\pi_*^{\cC}X,\pi_*^{\cC}Y) \implies \pi_{p+q}^{\cC}(X \otimes Y),
\]
where the bigraded Tor group $\Tor^{\pi_*^{\cC}A}_{*,*}$ is computed in the category of graded $\pi_*A$-modules. 
\end{cor}
\begin{proof}
By \Cref{prop:moritacell} and the observation that the notions of homotopy groups in $\cC^{\cell}$ and $\Mod_R$ agree under the equivalence, it suffices to construct the Tor spectral sequence 
\[
E_{p,q}^2 \cong \Tor_{p,q}^{\pi_*R}(\pi_*M,\pi_*N) \implies \pi_{p+q}(M \otimes_R N)
\]
for all $M,N \in \Mod_R$. A proof for arbitrary $\bE_1$-ring spectra can be found in \cite[Thm.~IV.4.1]{ekmm} and \cite[7.2.1.19]{ha}. 
\end{proof}

Specializing this result to the motivic homotopy category, we recover a variant of the spectral sequence constructed by Dugger and Isaksen in \cite[Prop.~7.7]{duggerisaksencell}.

\begin{rem}
More generally, we obtain a notion of $\cS$-cellularity for any finite set of compact objects $\cS \subseteq \cC^{\omega}$. The same arguments apply to this more general situation, giving rise to a notion of $\cS$-indexed homotopy groups 
\[
\xymatrix{\pi_*^{\cS}\colon \cC \ar[r] & \Mod_{R_{\cS}},}
\]
where $R_{\cS} = \End_{\cC}(\oplus_{S\in\cS}S)$. Note, however, that the equivalence given in \Cref{prop:moritacell} might not be symmetric monoidal, since in general $R_{\cS}$ has only the structure of an $\bE_{1}$-ring spectrum. 
\end{rem}

\section{Local duality for ring spectra}\label{sec:ringspectra}

\subsection{Module categories and the local to global principle}
Given the equivalence between cellular objects and ring spectra via derived Morita theory we study local homology and cohomology for ring spectra in this section. The results mentioned in \Cref{ssec:cell} actually hold under slightly weaker hypothesis, see \cite[Prop.~7.1.2.6 and Prop.~7.1.2.7]{ha} and also \cite[Thm.~3.1.1]{schwedeshipley_modules}. Recall that if $R$ is an $\mathbb{E}_k$-ring spectrum, the $\infty$-category $\Mod_R$ of (left) $R$-modules inherits an $\mathbb{E}_{k-1}$-monoidal structure \cite[Sec.~4.8.3]{ha}.
\begin{thm}[Schwede--Shipley, Lurie]
	Suppose that $(\cC,\otimes,A)$ is a stable, presentable, $\mathbb{E}_{k-1}$-monoidal $\infty$-category for $1 \le k \le \infty$, such that the tensor product functor preserves colimits separately in each variable. Then there is an equivalence $\cC \simeq \Mod_R$ of $\mathbb{E}_{k-1}$-monoidal $\infty$-categories if and only if $A$ is a compact generator of $\cC$, where $R \simeq \End_{\cC}(A)$ is naturally an $\mathbb{E}_{k}$-ring spectrum. 
\end{thm}
In general if $A$ is compact, but not a compact generator, then $\Mod_R$ is a colocalization of $\cC$, and is equivalent (still as  $\mathbb{E}_{k-1}$-monoidal $\infty$-categories) to the subcategory of cellular objects of $\cC$, as studied more thoroughly in \Cref{ssec:cell}.  

We hence let $R$ denote an arbitrary $\mathbb{E}_{\infty}$-ring spectrum; the theory works more generally for $\mathbb{E}_k$ ring spectra for $k \ge 2$, but we do not consider this case. The $\infty$-category $\Mod_R$ is a compactly generated (the compact generator is $R$ itself) $\mathbb{E}_{\infty}$-monoidal, presentable $\infty$-category. Moreover our assumptions imply that $\pi_*R$ is a graded-commutative ring, and that if $r \in \pi_*R$ is homogeneous of degree $d$, then the map $M \xr{r} \Sigma^{-d} M$ is a map of $R$-modules. We assume additionally that $\pi_*R$ is (graded) Noetherian, i.e. satisfies the ascending chain condition for homogeneous ideals, or equivalently every homogeneous ideal of $\pi_*R$ is finitely generated.    

Given a homogeneous ideal $\frak p \in \pi_*R$, we wish to build a Koszul object $\kos{R}{\frak p}$, and we do this in stages. We base this on \cite[Sec.~5]{benson_local_cohom_2008}, however we note it is the analogue of the classical construction from commutative algebra. Suppose that $a \in \pi_dR$, and let $|a| = -d$.\footnote{The unusual choice of grading is because we will work with the coconnective ring spectrum $R = C^*(BG,k)$ in the sequel.} For any $M \in \Mod_R$ we then define by $\kos{M}{a}$ by the fiber sequence
\[
 M \xr{a} \Sigma^{d} M \to \kos{M}{a}. 
\]
For a sequence of elements $\mathbf{a} = (a_1, \ldots, a_n)$ we define the Koszul complex $\kos{M}{\mathbf{a}}$ by iterating this procedure; in particular we set $\kos{M}{\mathbf{a}} = M_n$, where $M_0 = M$ and $M_i = \kos{M_{i-1}}{a_i}$ for $i \ge 1$.

\begin{lem}\label{lem:kos_tensor_decom}
	There is an equivalence 
	\[
\kos{M}{\mathbf{a}} \simeq M \otimes (\kos{R}{a_1}) \otimes (\kos{R}{a_2}) \otimes \cdots \otimes (\kos{R}{a_n}).
	\]
\end{lem}
\begin{proof}
	This is clear from the definitions. 
\end{proof}
If $\frak{p}$ is a finitely-generated ideal in $\pi_*R$, we let $\kos{M}{\frak{p}}$ denote the Koszul object based on some generating set of $\frak{p}$; note that the Koszul object itself is dependent on the choice of generators, but the localizing subcategory it generates depends only on the radical of the ideal, see \cite[Prop.~2.11]{benson_stratifying_2011} (in the language of \cite{bhv} $\kos{R}{\frak p}$ is equivalent to a shift of the object denoted $\Kos_1(\frak p)$). 

Given a set of homogeneous elements $T$ in $\pi_*R$, one can form the homogeneous localization $T^{-1}\pi_*R$. This procedure can be lifted to the spectrum level in the following sense. 
\begin{thm}[Finite algebraic localization]
For any set $T$ of homogeneous elements in $\pi_*R$ there is a finite localization $L_T$, and a natural equivalence $\pi_*(L_TM) \cong T^{-1}\pi_*M$ for any $M \in \Mod_R$. 
\end{thm}
\begin{proof}
This is proved in \cite[Ch.~V.1]{ekmm}. For a proof in the spirit of our techniques one can use the approach of \cite[Thm.~3.3.7]{hps_axiomatic}, which also makes it clear that $L_T$ is a finite localization. Let $\cT$ denote the localizing subcategory generated by $\{ \kos{R}{t} \mid t \in T\}$.  The functor $L_T$ arises from the local duality context $(\Mod_R,\cT)$ as the localization functor whose category of acyclics is $\cT$. We refer the reader to \cite[Thm.~3.3.7]{hps_axiomatic} for the proof of the claim that $\pi_*(L_TX) \cong T^{-1}\pi_*X$. 
\end{proof}
Taking $T = \pi_*R \setminus \frak p$, for any $\frak p \in \Spec \pi_*R$, we get \cite[Prop.~6.0.7]{hps_axiomatic}.
\begin{prop}\label{prop:hpsplocal}
For each $\frak p$ in $\Spec \pi_*R$ there is a smashing localization functor $L_{\frak p}$ with the property that $\pi_*(L_{\frak p}M) \cong (\pi_*M)_{\frak p}$ for any $M \in \Mod_R$. 
\end{prop}
We will often write $M_{\frak p}$ for $L_{\frak p}M$, and write $\Mod_{R_{\frak p}}$ for the essential image of $L_{\frak p}$. We also denote by $\Delta^{\frak p}$ the colocalization functor which is right adjoint to $L_{\frak p}$ when they are thought of as endofunctors of $\Mod_R$. 
\begin{rem}
The use of $L_{\frak p}$ conflicts with the notation in \cite[Thm.~3.6]{bhv}, however we do not use that notation in this paper so no confusion should arise. 
\end{rem}
Given a homogenous ideal $\frak p \in \Spec \pi_*R$, let $(\kos{R}{\frak p})_{\frak p} \simeq \kos{R_{\frak p}}{\frak p}$ denote the localization of the Koszul object. We then have the following crucial result of Hovey, Palmieri, and Strickland \cite[Thm.~6.1.9]{hps_axiomatic} which we will use repeatedly. 
\begin{thm}[Local-to-global principle]\label{thm:hps_bousfield}
There is an equality of Bousfield classes 
\[
\langle R \rangle = \coprod_{\frak p \in \Spec \pi_*R} \langle \kos{R_{\frak p}}{\frak p} \rangle.
\]
In particular, for $M \in \Mod_R$ we have $M \simeq 0$ if and only if $M \otimes \kos{R_{\frak p}}{\frak p} \simeq 0$ for all $\frak p \in \Spec \pi_*R$. 
\end{thm}
\begin{rem}
	In the context of triangulated categories with an action by a commutative ring $R$, Benson, Iyengar, and Krause introduce a local-to-global principle \cite[Sec.~3]{benson_stratifying_2011}. In the case of ring spectra, the local-to-global principle holds if for each object $M \in \Mod_R$ there is an equality
	\[
\Loc(M) = \Loc(\{ \Gamma_{\fp}M \mid \fp \in \Spec \pi_*R \}),
	\]
	where $\Gamma_{\fp}$ is a certain local cohomology functor which we will study more thoroughly in the sequel. Note that this holds automatically in $\Mod_R$, by combining \cite[Prop.~6.3.2]{hps_axiomatic} and \cite[Thm.~3.1(2)]{benson_stratifying_2011}.
\end{rem}

\subsection{Local homology and cohomology functors}\label{sec:localcohomrs}
We begin by reviewing the construction of local homology and cohomology functors for ring spectra. This is mainly a review of \cite[Sec.~3]{bhv} (which in turn is based on \cite{greenleesmay_completions}), although we will generalize it slightly. To start, let $R$ be an $\mathbb{E}_\infty$-ring spectrum with $\pi_*R$ Noetherian.

\begin{defn}\label{def:specclosed}A subset $\cal U \subseteq \Spec \pi_*R$ is called specialization closed if $\frak p \in \cal U$ and $\frak p \subseteq \frak q \in \Spec \pi_*R$, implies that $\frak q \in \cal{U}$.
\end{defn} 

Equivalently, these are unions of Zariski closed subsets of $\Spec \pi_*R$. For instance, for any homogeneous ideal $\frak p$, the set $\mathcal V(\frak p)=\{\frak q \in \Spec \pi_*R  \mid  \frak p \subseteq \frak q \}$ is specialization closed. Given any specialization closed subset $\mathcal V$, we consider the full subcategory of $\Mod_R$
\[
\Mod_R^{\mathcal V -\text{tors}}:=\Locid{}(\kos{R}{\frak p} \mid  \frak p \in \mathcal{V} ) = \Loc(\kos{R}{\frak p} \mid  \frak p \in \mathcal{V} ),
\]
where the latter equality comes from the observation that every localizing subcategory is automatically a localizing ideal, since $\Mod_R$ is compactly generated by $R$ \cite[Lem.~1.4.6]{hps_axiomatic}. This is by definition a compactly generated subcategory, and so \Cref{thm:abstractlocalduality,lem:smashing} yields:

\begin{thm}\label{thm:localduality}
 There exists a quadruple $(\Gamma_{\cV},L_{\cV},\Delta_{\cV},\Lambda_{\cV})$ of endofunctors on $\Mod_R$ satisfying the properties of \Cref{thm:abstractlocalduality}. In addition, $\Gamma_{\mathcal V}$ and $L_{\cV}$ are smashing.
\end{thm}

\begin{rem}[Recollements] By \Cref{prop:abstractrecollement} there are recollements associated to the above local duality contexts. In particular we have
\[
\xymatrixcolsep{3pc}
\xymatrix{\Mod_{R}^{\cV-\text{loc}} \ar[r]|-{\iota_{\text{loc}}} & \Mod_R \ar[r]|-{\Gamma_{\cal V}} \ar@<1.2ex>[l]^-{L_{\cal V}} \ar@<-1.2ex>[l]_-{\Delta^{\cal V}} & \Mod_R^{\cV-\text{tors}}, \ar@<1.2ex>[l]^-{\iota_{\text{tors}}} \ar@<-1.2ex>[l]_-{\iota_{\text{comp}}\Lambda'}
}
\xymatrixcolsep{3pc}
\xymatrix{
  \Mod_R^{\cV-\text{loc}} \ar[r]|-{\iota_{\text{loc}}} & \Mod_R \ar[r]|-{\Lambda^{\cV}} \ar@<1.2ex>[l]^-{L_{\cV}} \ar@<-1.2ex>[l]_-{\Delta^{\cV}} & \Mod_R^{\cV-\text{comp}} \ar@<1.2ex>[l]^-{\iota_{\text{tors}}\Gamma'} \ar@<-1.2ex>[l]_-{\iota_{\text{comp}}}}
\]
where $\Lambda' = \Lambda^{\cV}\iota_{\text{tors}}$ and $\Gamma' = \Gamma_{\cV}\iota_{\text{comp}}$. The first recollement is a particular case of \cite[Thm.~6.7]{benson_local_cohom_2008}. 
	\end{rem}
\begin{exmp}
Consider the specialization closed subset $\cal{Z}(\frak p) = \{ \frak q \in \Spec \pi_*R \mid \frak q \not \subseteq \frak p \}$. It is a result of Benson, Iyengar, and Krause \cite[Thm.~4.7]{benson_local_cohom_2008} that $L_{\cal{Z}(\fp)}$ has the property that $\pi_*(L_{\cal{Z}(\fp)}M) \cong (\pi_*M)_{\fp}$ for any $M \in \Mod_R$. We will see in \Cref{cor:bikcompare} that $L_{\cal Z}(\fp)$ is equivalent to $L_{\fp}$ as defined in \Cref{prop:hpsplocal}. 
\end{exmp}

This theorem is a generalization of \cite[Thm.~3.6]{bhv} where we considered local duality contexts of the form $(\Mod_R,\Loc(\kos{R}{\frak p}))$ for ideals $\frak p \in \Spec \pi_*R$. In fact, these are a special case of the above, applied to the specialization closed subset $\cV = \cV(\frak p) = \{ \frak q \in \Spec \pi_*R \mid \frak p \subseteq \frak q \}$. 

\begin{prop}\label{prop:loccomparasion}
	For any homogeneous ideal $\frak p \in \Spec \pi_*R$ the local duality contexts
	\[(\Mod_R,\Loc(\kos{R}{\frak p})) \quad \text{ and } \quad (\Mod_R,\Loc(\kos{R}{\frak q} \mid \frak q \in \cal{V}(\frak p)))\]
	are equivalent. 	
\end{prop}
\begin{proof}
	Clearly it suffices to show that $\Loc(\kos{R}{\frak p}) = \Loc(\kos{R}{\frak q} \mid \frak q \in \cal{V}(\frak p))$. The argument now proceeds as in \cite[Prop.~3.1]{hov_pal_gal} which we include for the benefit of the reader. 

	 First, suppose that $\frak q \in \cal{V}(\frak p)$ so that $\frak p \subseteq \frak q$. By \cite[Lem.~6.0.9]{hps_axiomatic} $\kos{R}{\frak q}$ even lies in the thick subcategory generated by $\kos{R}{\frak p}$, so that in particular we have $\Loc(\kos{R}{\frak q} \mid \frak q \in \cal{V}(\frak p)) \subseteq \Loc(\kos{R}{\frak p})$. 

	 For the other inclusion, it suffices to show that $L_{\cV(\frak p)}\kos{R}{\frak p} \simeq L_{\cV(\frak p)}R \otimes \kos{R}{\frak p} = 0$, or equivalently, by applying \Cref{thm:hps_bousfield}, that $L_{\cV(\frak p)}R \otimes \kos{R}{\frak p} \otimes \kos{R}{\frak q} \otimes R_{\frak q} = 0$, for all $\frak q \in \Spec \pi_*R$. If $ \frak q \not \supseteq \frak p$, then $\kos{R}{\frak p} \otimes R_{\frak q} = 0$ by \cite[Prop.~6.1.7(d)]{hps_axiomatic}. On the other hand, if $\frak q \frak \supseteq \frak p$, then $\frak q \in \cV(\frak p)$, and $L_{\cV(\frak p)}\kos{R}{\frak q} = 0$, and we are done.
\end{proof}
Let $a \in \pi_*R$ be a homogeneous element of degree $-d$ and consider the diagram
\[
\begin{gathered}
\xymatrix{
M \ar@{=}[r]\ar[d]_{a} & M \ar@{=}[r]\ar[d]_{a^2} & M \ar@{=}[r]\ar[d]_{a^3} & \cdots\\
\Sigma^{d} M \ar[r]^{a} \ar[d] & \ar[d] \Sigma^{2d} M \ar[r]^{a} &\ar[d] \Sigma^{3d} M  \ar[r]^{a} & \cdots \\
	\kos{M}{a} \ar[r] & \kos{M}{a^2}\ar[r] & \kos{M}{a^3}\ar[r] & \cdots
}
\end{gathered}
\]
where the morphisms in the last horizontal sequence are the ones induced by the commutativity of the squares above. For an ideal $\frak p = (p_1,\ldots,p_n)$ we let $\kos{M}{\frak p^s} = M \otimes \kos{R}{p_1^s} \otimes \cdots \otimes \kos{R}{p_n^s}$. This is a shift of the object denoted  $\Kos_s(\frak p)$ in \cite{bhv} (in particular $\Kos_s(\frak p) \simeq \Sigma^{s(|p_1|+\cdots+|p_n|)-n}\kos{R}{\frak p^s}$), and applying \cite[Thm.~3.6]{bhv} we get the following. 
\begin{cor}\label{cor:formulas}
	For $M \in \Mod_R$ and $-d = |p_1|+\cdots + |p_n|$ there are natural equivalences
	\[
\Gamma_{\cV(\frak p)}M \simeq \colim_s \Sigma^{-sd-n}\kos{R}{\frak p^s} \otimes M \quad \text{and} \quad \Lambda^{\cV(\frak p)}M \simeq \lim_s  \kos{R}{\frak p^s} \otimes M.
	\]
\end{cor}

We have the following spectral sequences due to Greenlees and May \cite{greenleesmay_completions}, which also appeared in \cite[Prop.~3.17]{bhv}. Here we write $H_{\frak p}^*$ for the local cohomology of a $\pi_*R$-module, $H_*^{\frak p}$ for the local homology, and $\cH^*_{\frak p}$ for the \v{C}ech cohomology groups; we refer the reader to \cite{greenleesmay_completions} for a convenient review of these functors. 
\begin{prop}\label{prop:homotopy_ss}
	Let $R$ be a ring spectrum with $\pi_*R$ Noetherian and $\frak p$ a finitely generated homogeneous ideal of $\pi_*R$. Let $M\in \Mod_{R}$. There are strongly convergent spectral sequences of $\pi_*R$-modules:
	\begin{enumerate}
		\item $E^2_{s,t}=(H^{-s}_{\frak p}(\pi_*M))_{t} \implies \pi_{s+t}(\Gamma_{\cV(\frak p)} M)$, with differentials $d^r\colon E^r_{s,t} \to E^r_{s-r,t+r-1}$,
		\item $E^2_{s,t}=(H_{s}^{\frak p}(\pi_*M))_{t} \implies \pi_{s+t}(\Lambda^{\cV(\frak p)} M)$, with differentials $d^r\colon E^r_{s,t} \to E^r_{s-r,t+r-1}$, and
		\item $E^2_{s,t}=(\cH^{-s}_{\frak p}(\pi_*M))_{t} \implies \pi_{s+t}(L_{\cV(\frak p)} M)$, with differentials $d^r\colon E^r_{s,t} \to E^r_{s-r,t+r-1}$.
	\end{enumerate}
\end{prop}
\begin{rem}
	If $\cV$ is an arbitrary specialization closed subset, then we expect a spectral sequence of the form
	\[
	E^2_{s,t}=(H^{-s}_{\cV}(\pi_*M))_{t} \implies \pi_{s+t}(\Gamma_{\cV} M),
	\]
	where $H^*_{\cV}$ refers to the local cohomology functor constructed in \cite[Eq.~(3.5)]{lipman}, however we have not checked the details.
\end{rem}

Given $\frak p \in \Spec \pi_*R$, we now define $\frak p$-local cohomology and homology functors. We will see in \Cref{cor:bikcompare} that these agree with those considered by Benson, Iyengar and Krause in \cite{benson_local_cohom_2008,benson_colocalizing_2012}. We remind the reader that $\Delta^{\frak p}$ denotes the functor that is right adjoint to $L_{\frak p}$ when they are considered as endofunctors of $\Mod_R$. 
\begin{defn}\label{def:pcohomology}
	For $\frak p \in \Spec \pi_*R$ and $M \in \Mod_R$ define the $\frak p$-local cohomology and homology functors, respectively, by $\Gamma_{\frak p}M = \Gamma_{\cV(\frak p)}M_{\frak p}$ and $\Lambda^{\frak p}M = \Lambda^{\cV(\frak p)}\Delta^{\frak p}M$.
\end{defn}
Let $\Mod_{R_{\frak p}}^{\frak p-\text{tors}}$ and $\Mod_{R_{\frak p}}^{\frak p-\text{comp}}$ denote the essential images of $\Gamma_{\frak p}$ and $\Lambda^{\frak p}$ respectively; note that these are both full subcategories of $\Mod_R$. Moreover, $\Mod_{R_{\frak p}}^{\frak p-\text{tors}}$ inherits a symmetric monoidal structure from $\Mod_R$, with unit $\Gamma_{\frak p}R$.

\begin{lem}
	The functor $\Gamma_{\frak p} \colon \Mod_R \to \Mod_{R_{\frak p}}^{\frak p-\text{tors}}$ is a smashing, symmetric monoidal functor, and $\Lambda^{\frak p}$ is right adjoint to $\Gamma_{\frak p}$ (considered as endofunctors of $\Mod_R$), so that there is an equivalence
	\[
	\Hom_R(\Gamma_{\frak p}M,N) \simeq \Hom_R(M,\Lambda^{\frak p}N)
\]
	for $M,N \in \Mod_R$. Moreover, for $-d = |p_1|+\cdots + |p_n|$, we have 	
\[
\Gamma_{\fp}M \simeq \colim_s \Sigma^{-sd-n}\kos{R_{\fp}}{\frak p^s} \otimes M \quad \text{and} \quad \Lambda^{\fp}M \simeq \lim_s  \kos{R}{\frak p^s} \otimes \Delta^{\fp}M,
\]

\end{lem}
\begin{proof}
	Since $\Gamma_{\frak p}$ is the composite of smashing functors, it too is smashing, and this implies that it is symmetric monoidal. That $\Lambda^{\frak p}$ is right adjoint to $\Gamma_{\frak p}$ follows  from the fact that $\Delta^{\frak p}$ is right adjoint to $L_{\frak p}$ and $\Lambda^{\cV(\frak p)}$ is right adjoint to $\Gamma_{\cV(\frak p)}$. The formulas are an immediate consequence of \Cref{cor:formulas}.
\end{proof}
The following is immediate from \Cref{prop:loccomparasion} and the observation that $(\kos{R}{\frak p})_{\frak p} \simeq \kos{R_{\frak p}}{\frak p}$. 
\begin{lem}\label{lem:genplocptors}
	For any $\frak p \in \Spec \pi_*R$ the full subcategory $\Mod_{R_{\frak p}}^{\frak p-\text{tors}}$ is equivalent to $\Loc(\kos{R_{\frak p}}{\frak p})$. 
\end{lem}

Note that $\kos{R_{\frak p}}{\frak p}$ need not be compact in $\Mod_R$, however $\kos{R_{\frak p}}{\frak p}$ is compact in $\Mod_{R_{\frak p}}$. Hence one can consider the local duality context $(\Mod_{R_{\frak p}},\Loc(\kos{R_{\frak p}}{\frak p}))$. Let $\Gamma'_{\cV(\frak p)}$ denote the associated local cohomology functor, so that $\Gamma'_{\cV(\frak p)}M \simeq \colim_s \Sigma^{-sd-n} R_{\fp}\mm \fp^s \otimes M$ via an argument similar to \Cref{cor:formulas}. Note that this formula makes sense for $M \in \Mod_R$ and not just $M \in \Mod_{R_{\fp}}$; indeed $\Gamma'_{\cV(\fp)}$ is simply the restriction of $\Gamma_{\cV(\fp)}$ to $\Mod_{R_{\fp}}$. 

We have spectral sequences for computing the homotopy groups of these functors.
\begin{prop}\label{prop:homotopy_ss_local}
	Let $R$ be a ring spectrum with $\pi_*R$ Noetherian and $\frak p$ a finitely generated homogeneous ideal of $\pi_*R$. Let $N\in \Mod_{R}$. There are strongly convergent spectral sequences of $\pi_*R$-modules:
	\begin{enumerate}
		\item $E^2_{s,t}=(H^{-s}_{\frak p}(\pi_*N)_{\frak p})_{t} \implies \pi_{s+t}(\Gamma_{\frak p} N)$, with differentials $d^r\colon E^r_{s,t} \to E^r_{s-r,t+r-1}$, and
		\item $E^2_{s,t}=(H_{s}^{\frak p}(\pi_*\Hom_R(R_{\frak p},N)))_{t} \implies \pi_{s+t}(\Lambda^{\frak p} N)$, with differentials $d^r\colon E^r_{s,t} \to E^r_{s-r,t+r-1}$.
	\end{enumerate}
\end{prop}
\begin{proof}
	First the first, use the spectral sequence \Cref{prop:homotopy_ss}(1) with $M = N_{\frak p}$, and use \Cref{prop:hpsplocal}. For the second, use \Cref{prop:homotopy_ss}(2) with $M = \Delta^{\frak p}N$. To identify the $E_2$-term, note that $\Delta^{\frak p}N \simeq \Hom_R(R,\Delta^{\frak p}N) \simeq \Hom_R(R_{\frak p},N)$, by adjointness.
\end{proof}
The following is also proved in \cite[Thm.~6.1.8]{hps_axiomatic}. We give a proof using the above spectral sequence. 
\begin{cor}\label{thm::hpsclassification}
	An $R$-module $M$ is in $\Mod_{R_{\frak p}}^{\frak p-\text{tors}}$ if and only if $\pi_*M$ is $\frak p$-local and $\frak p$-torsion. 
\end{cor}
\begin{proof}
	Note that both conditions imply that $M$ is $\frak p$-local (i.e., $M_{\frak p} \simeq M$), so we may as well assume that $\pi_*R$ is local with maximal ideal $\frak p$. Thus, we can drop the assumption of $\frak p$-locality, and show that $M$ is in $\Mod_R^{\frak p-\text{tors}}$ if and only if $\pi_*R$ is $\frak p$-torsion. 

	If $\pi_*M$ is $\frak p$-torsion, then $H^0_{\frak p}(\pi_*M) \cong \pi_*M$, and the higher local cohomology groups are 0. Thus, the spectral sequence of \Cref{prop:homotopy_ss}(1) collapses to give an isomorphism $\pi_*(\Gamma_{\vP}M) \xr{\simeq} \pi_*M$, so that $M$ is in $\Mod_R^{\frak p-\text{tors}}$. 

	For the other direction suppose that $N \in \Mod_R$ is arbitrary. The $E_2$-term of the spectral sequence computing $\pi_*(\Gamma_{\vP}N)$ is all $\frak p$-torsion (see, for example, \cite[Rem.~2.1.3]{broadmann_sharp}). This implies that the $E_\infty$-page is also all $\frak p$-torsion, and since the spectral sequence of \Cref{prop:homotopy_ss}(1) has a horizontal vanishing line, the abutment $\pi_*(\Gamma_{\vP}N)$ is also $\frak p$-torsion. In our case this implies that $\pi_*(\Gamma_{\vP}M) \cong \pi_*(M)$ is $\frak p$-torsion. 
\end{proof}
\subsection{Comparison to the Benson--Iyengar--Krause functors}
	For any specialization closed subset $\cV \subset \Spec \pi_*R$ (see \Cref{def:specclosed}), Benson, Iyengar, and Krause construct a localization functor $L_{\cV} \colon \Mod_R \to \Mod_R$ whose kernel is precisely 
	\[
T_{\cV} = \{X \in \Mod_R \mid (\pi_*X)_{\frak p} = 0 \text{ for all } \frak p \in \Spec \pi_*R \setminus \cV \},
	\]
see \cite[Sec.~3]{benson_local_cohom_2008} where such functors are constructed more generally for compactly generated triangulated categories with small coproducts with an action by a Noetherian ring $R$. The subcategory $T_{\cV}$ is localizing \cite[Lem.~4.3]{benson_local_cohom_2008} and by Thm.~6.4 of \emph{loc.~cit.~}~agrees with $\Loc(\kos{R}{\frak p} \mid  \frak p \in \mathcal{V} )$. We give a direct proof of that fact here. 
 
\begin{prop}\label{prop:bikcomparespecial}
	For any specialization closed subset $\cV$ the localizing subcategories $T_{\cV}$ and $\Loc(\kos{R}{\frak p} \mid \frak p \in \cV)$ are equivalent. 
\end{prop}
\begin{proof}
	Note that for any $\frak p \in \cV$, we have $\cV(\frak p) \subseteq \cV$ since $\cV$ is specialization closed. It follows as in the proof of \Cref{thm::hpsclassification} that the homotopy groups $\pi_*(\kos{R}{\frak p})$ are $\frak{p}$-torsion. We see then that $(\pi_*\kos{R}{\frak p})_{p} = 0$ for all $p \in \frak{p}$, and hence that $(\pi_*\kos{R}{\frak p})_{\frak q} = 0$ for all $\frak p \not \subset \frak q$. It follows that $\kos{R}{\frak p} \in T_{\cV(\frak p)} \subseteq T_{\cV}$. Since this is true for each $\frak p \in \cV$ we see that $\Loc(\kos{R}{\frak p} \mid \frak p \in \cV) \subseteq T_{\cV}$. 

	For the converse suppose that $(\pi_*M)_{\frak q} = 0$ for all $\frak q \in \Spec \pi_*R \setminus \cV$. It suffices to show that $L_{\cV}M \simeq L_{\cV}R \otimes M \simeq 0$. Applying \Cref{thm:hps_bousfield} again it suffices to show that $L_{\cV}R \otimes M \otimes \kos{R}{\frak a} \otimes R_{\frak a} \simeq L_{\cV}(\kos{R}{\frak a}) \otimes M_{\frak a} \simeq 0$ for all $\frak a \in \Spec \pi_*R$. If $\frak a \in \cV$ then $L_{\cV}(\kos{R}{\frak a}) = 0$, so assume that $\frak a \not \in \cV$; by assumption then, $(\pi_*M)_{\frak a} = 0$.  There is a spectral sequence \cite[Thm.~IV.4.1]{ekmm}
	\[
E^2_{p,q} = \Tor^{R_*}_{p,q}(\pi_*L_{\cV}(\kos{R}{\frak a}),\pi_*(M_{\frak a})) \implies \pi_*(L_{\cV}(\kos{R}{\frak a}) \otimes M_{\frak a}). 
	\]
	Since $\pi_*(M_{\frak a}) \cong (\pi_*M)_{\frak a} = 0$ (see \Cref{prop:hpsplocal}) the spectral sequence shows that $L_{\cV}(\kos{R}{\frak a}) \otimes M_{a} \simeq 0$ as required. 
\end{proof}
We can now provide the proof of the claimed statement before \Cref{def:pcohomology}. We recall that Benson, Iyengar, and Krause construct functors $\Gamma_{\frak p}$ and $\Lambda^{\frak p}$ for a compactly generated triangulated category $\cal{T}$ with set-indexed coproducts with an action by a Noetherian ring $R$.

\begin{cor}\label{cor:bikcompare}
	The functors $\Gamma_{\frak p}$ and $\Lambda^{\frak p}$ agree with the functors with the same name constructed by Benson, Iyengar, and Krause for the case $\cC = \Mod_R$. 
\end{cor}
\begin{proof}
Note the the functor $L_{\frak p}$ has category of acyclics $T_{Z(\frak p)}$ where $\cal{Z}(\frak p) = \{ q \in \Spec \pi_*R \mid \frak q \not \subseteq \frak p \}$, and hence agrees with $L_{\mathcal{Z}(\frak p)}$. Then, by \Cref{prop:bikcomparespecial} the functors $\Gamma_{V(\frak p)}$ and $L_{\cal{Z}(\frak p)}$ agree with those given the same name in \cite{benson_local_cohom_2008}; this proves the result for $\Gamma_{\frak p}$.  The local homology functor $\Lambda^{\frak p}$ is constructed in \cite{benson_colocalizing_2012} as the right adjoint of $\Gamma_{\frak p}$; it is proved in \Cref{lem:genplocptors} that $\Gamma_{\frak p}$ and $\Lambda^{\frak p}$, as constructed in this paper, form an adjoint pair. By uniqueness of adjoints $\Lambda^{\frak p}$ agrees with the functor with the same name constructed by Benson, Iyengar, and Krause.
\end{proof}

\section{\cGorenstein ring spectra}\label{sec:gorenstein}
	In \cite{bg_localduality} Benson and Greenlees prove a conjecture due to Benson \cite{benson_moduleswithinjcohom}, which amounts to a local duality theorem in modular representation theory. Let $G$ be a finite group and $k$ a field of characteristic $p>0$. Benson and Greenlees prove that a certain $R = C^*(BG,k)$-module $T_R(I_{\frak p})$ is isomorphic to a shift of $\Gamma_{\fp}R$. One goal of this section is to give an alternative explanation for this equivalence. In fact, we will generalize the result of Benson and Greenlees to a larger class of ring spectra.

\subsection{Matlis duality and Brown--Comenetz duality}
Let $I$ be an injective $R_*$-module; we introduce certain $R$-modules $T_R(I)$ with the property that $\pi_*\Hom_R(-,T_R(I))$ is isomorphic to $\Hom_{\pi_*R}(\pi_*(-),I)$. 
To construct these, note that for $I \in \Inj_{\pi_*R}$, the functor
\[
\xymatrix{\Hom_{\pi_*R}(\pi_*(-),I)\colon \Mod_{\pi_*R} \ar[r] & \Mod_{\Z}}
\]
is exact and therefore representable by an object $T_R(I) \in \Mod_R$ by Brown representability. This construction is natural, hence induces a functor
\[
\xymatrix{T_R\colon\Inj_{\pi_*R} \ar[r] & \Ho(\Mod_R),}
\]
to the homotopy category of $R$-modules. This functor appears to have first been defined in \cite[Sec.~6]{hps_axiomatic} and then studied more thoroughly in the stable module category for a finite group in \cite{bk_pureinjectives}. The functor $T$ has the following universal property; for any injective $\pi_*R$-module $I$ and $M \in \Mod_R$, there is an isomorphism $\pi_0\Hom_{R}(M,T_R(I)) \cong \Hom^0_{\pi_*R}(\pi_*M,I)$. A standard dimension shifting argument (e.g.~\cite[Lem.~3.2]{bk_pureinjectives} or \cite[Lem.~5.3]{bg_localduality}) shows that this extends to an isomorphism of graded $\pi_*R$-modules
\begin{equation}\label{eq:liftdef}
\pi_*\Hom_{R}(M,T_R(I)) \cong \Hom_{\pi_*R}(\pi_*M,I). 
\end{equation}
In particular, taking $M = R$ we get:
\begin{lem}\label{lem:cohom}
	Suppose $I$ is an injective $\pi_*R$-module. Then there is a natural isomorphism $\pi_*T_R(I) \cong I$. 
\end{lem}
It follows that $\pi_*\Hom_{R}(T_R(I'),T_R(I)) \cong \Hom_{\pi_*R}(I',I)$ so that $T_R$ is fully faithful. We also have the following simple observation. 
\begin{lem}\label{lem:homotopycheck}
Let $I$ be an injective $\pi_*R$-module. If $M \in \Mod_R$ is such that $\pi_*M \cong I$, then $M \simeq T_R(I)$.	
\end{lem}
\begin{proof}
	The equivalence $\pi_*M \cong I$ lifts to a morphism $M \to T_R(I)$ by \eqref{eq:liftdef}, which is an isomorphism on homotopy. 
\end{proof}

For the following, let $I_{\frak p}$ denote the injective hull of $\pi_*R/ \frak p$. 
\begin{lem}\label{lem:lifttorlocal}
	For each $\frak p \in \Spec \pi_*R$, we have $T_R(I_{\frak p}) \in \Mod_{R_{\frak p}}^{\frak p-\text{tors}}$. 
	\end{lem}
\begin{proof}
	Note first that $\pi_*T_R(I_{\frak p}) \cong I_{\frak p}$ is both $\frak p$-local and $\frak p$-torsion. Indeed the torsion statement is proved in \cite[Ch.~18]{matsumura}, while it is easy to see (for example, \cite[Lem.~2.1]{benson_local_cohom_2008}) that 
	\[
(I_\fp)_{\frak q} = \begin{cases}
I_\fp \quad &\text{if } \frak q \in \cV(\frak p), \\
0 & \text {otherwise.}
\end{cases}
	\]
The result then follows from \Cref{thm::hpsclassification}. 
\end{proof}
For the following recall that a graded ring is called local if it has a unique maximal homogeneous ideal. 
\begin{defn}
We say that a ring spectrum $R$ with $\pi_*R$ local Noetherian of dimension $n$ is algebraically Gorenstein of shift $\nu$ if $\pi_*R$ is a graded Gorenstein ring; that is, the local cohomology $H_{\frak m}^i(\pi_*R)$ is non-zero only when $i = n$ and $(H^n_{\frak m}(\pi_*R))_t \cong (I_{\frak m})_{t-\nu-n}$. If $\pi_*R$ is non-local, then it is algebraically Gorenstein of shift $\nu$ if its localization at each maximal ideal is algebraically Gorenstein of shift $\nu$ in the above sense.
\end{defn}

 \begin{lem}
  	Let $p \in \Spec \pi_*R$ be an ideal of dimension $d$ (i.e., $d$ is the Krull dimension of $\pi_*R/\frak p$). If $R$ is algebraically Gorenstein with shift $\nu$, then $R_{\fp}$ is algebraically Gorenstein of shift $\nu+d$. 
 \end{lem}
\begin{proof}
	That $R_{\fp}$ is Gorenstein is well-known, see, for example, \cite[Thm.~18.2]{matsumura}. Thus we just need to calculate the shift.  First, note that $R_{\fp}$ has dimension $n-d$. By assumption, we have
\[
(H^n_{\frak m}(\pi_*R))_t \cong (I_{\frak m})_{t-\nu-n}
\]
and by the same argument in the proof of Lem.~7.1 of \cite{green_lyub} we get
\[
(H^{n-d}_{\frak p}(\pi_*R)_{\fp})_t \cong (I_{\fp})_{t-(\nu+d)-(n-d)}.\qedhere 
\]
\end{proof}
\begin{prop}\label{prop:grosshopkinsgorenstein}
	 Suppose $R$ is algebraically Gorenstein of shift $\nu$, and suppose $\frak \p \in \Spec \pi_*R$ has dimension $d$. Then, there is an equivalence $\Gamma_{\frak p}R \simeq \Sigma^{\nu+d}T_R(I_{\frak p})$. 
\end{prop}
\begin{proof}
		The spectral sequence \Cref{prop:homotopy_ss_local}(1) computing $\Gamma_{\frak p}R$ takes the form
	\[
(H^{s}_{\frak p}(\pi_*R)_{\frak p})_{t+s} \implies \pi_{t}(\Gamma_{\frak p}R),
	\]
	with $(H^{\ast}_{\frak p}(\pi_*R)_{\frak p})_{t+n-d} \cong (I_{\frak p})_{t-\nu-d}$ when $\ast = n-d$ and is 0 otherwise. Hence the spectral sequence collapses to show that $(I_{\frak p})_{t-\nu-d} \cong \pi_{t}(\Gamma_{\frak p}R)$.
	 The result follows from \Cref{lem:homotopycheck}.%\footnote{Note that our  grading conventions imply that $\pi_t\Sigma^{\nu-d}T_R(I_{\frak p}) \cong (I_{\frak p})_{t-\nu+d}$.}  
\end{proof}
\begin{exmp}\label{exmp:Gorenstein}
	\begin{enumerate}
		\item Suppose $G$ is a finite group and $H^{-\ast}(G,k) \cong \pi_*C^*(B,k)$ is a Gorenstein ring. Then $C^*(BG,k)$ is algebraically Gorenstein of shift $0$. This follows from the collapsing of Greenlees' spectral sequence \cite[Thm.~2.1]{green_1995}
		\[
(H_{\frak m}^{s}H^*(G,k))_{-t} \implies (I_{\fp})_{t-s}.
		\]
		\item Suppose $G$ is a compact Lie group of dimension $w$, such that $H^*(G,k)$ is Gorenstein, and the adjoint representation of $G$ is orientable over $k$. Then, using the Benson--Greenlees spectral sequence \cite[Cor.~5.2]{ben_green_ss}, we see $C^*(BG,k)$ is algebraically Gorenstein of shift $w$. 

	\end{enumerate}
\end{exmp}

Suppose that $f \colon R \to S$ is a morphism of ring spectra. Recall that this gives rise to adjoint pairs $(f_*,f^*)$ and $(f^*,f_!)$, where $f^* \colon \Mod_S \to \Mod_R$ is restriction of scalars, $f_* \colon \Mod_R \to \Mod_S, M \mapsto S \otimes_R M$ is extension of scalars, and $f_! \colon \Mod_R \to \Mod_S, M \mapsto \Hom_R(S,M)$ is coinduction.

 Let $r_{R,S}(I)= \Hom_{\pi_*R}(\pi_*S,I)$ for each injective $\pi_*R$-module $I$; note that $r_{R,S}(I)$ is an injective $\pi_*S$-module. 
The following is the analog of \cite[Prop.~7.1]{bk_pureinjectives}, except here the role of restriction is played by coinduction.  
\begin{prop}\label{prop:liftsplit}
	Let $f \colon R \to S$ be a morphism of ring spectra. If $I$ is an injective $\pi_*R$-module, then there is a weak equivalence $f_!T_R(I) \simeq T_S(r_{R,S}(I))$. 
	\end{prop}
\begin{proof}
	Let $M \in \Mod_S$. There are equivalences
	\[
\begin{split}
	\pi_*\Hom_S(M,f_!T_R(I)) &\cong \pi_* \Hom_R(f^*M,T_R(I))\\
						& \cong \Hom_{\pi_*R}(\pi_*(f^*M),I) \\
						& \cong \Hom_{\pi_*S}(\pi_*M,r_{R,S}(I))
\end{split}
	\]
	where the last step follows from adjunction and the observation that $\pi_*(f^*M) \cong \pi_*\Hom_{R}(R,f^*M) \cong \pi_*\Hom_S(S,M) \cong \pi_*M$. Since $r_{R,S}(I)$ is an injective $\pi_*S$-module we have
	\[
\Hom_{\pi_*S}(\pi_*M,r_{R,S}(I)) \cong \pi_*\Hom_S(M,T_S(r_{R,S}(I))). 
	\]
Taking $M = S$ and applying \Cref{lem:homotopycheck} gives the result. 
\end{proof}
In \cite{green_lyub} Greenlees and Lyubeznik introduced a way to localize local cohomology modules at some non-maximal prime ideal $\fp$; note that $H_{\frak m}^*$ is always $\frak m$-torsion, and so the naive approach of directly localizing at $\frak p$ does not work. The approach of Greenlees and Lyubeznik is to first dualize, then localize, and then dualize again. We assume for simplicity that $\pi_*R$ is a graded local ring with maximal ideal $\frak m$.
\begin{defn}
	Let $\frak p$ be a homogeneous ideal in $\Spec(\pi_*R)$, and suppose $M$ is a $(\pi_*R)_{\frak p}$-module. Let $I_{\fp }$ denote the injective hull of $\pi_*R/\fp$. The Matlis dual $D_{\frak p}M$ is defined by
	\[
D_{\frak p}M = \Hom_{(\pi_*R)_{\frak p}}(M,I_{\frak p}).
	\]
\end{defn}
Greenlees and Lyubenzik introduced a functor, called dual localization and denoted $\cal{L}_{\frak p} \colon \Mod_{\pi_*R} \to \Mod_{(\pi_*R)_{\frak p}}$, which is defined as the composite
\[
\Mod_{\pi_*R} \xr{D_{\frak m}} \Mod^{\text{op}}_{\pi_*R} \xr{(-)_{\frak p}} \Mod^{\text{op}}_{(\pi_*R)_{\frak p}} \xr{D_{\frak p}} \Mod_{(\pi_*R)_{\frak p}.}
\]
For example, if $\frak p$ has dimension $d$ and $M$ is a finitely-generated $\pi_*R$-module, then by \cite[Lem.~2.5]{green_lyub} 
\begin{equation}\label{eq:matlisdual}
\cal{L}_{\frak p}H_{\frak m}^i(M) \cong H_{\frak p}^{i-d}(M_{\frak p}). 
\end{equation}

There is an obvious way to define a lift of this functor to topology. First, we start with the analog of the Matlis dual of a module. 
\begin{defn}
	We define the Brown--Comenetz dual $\mathbb{D}_{T_R(I)} \colon \Mod_R \to \Mod_R^{\text{op}}$ by setting $\mathbb{D}_{T_R(I)}(M) = \Hom_R(M,T_R(I))$, which has a natural structure of an $R$-module. For ease of notation we will write $\mathbb{D}_{\frak p}$ for $\mathbb{D}_{T_R(I_{\frak p})}$. 
\end{defn}

We can then define a functor $\mathbb{L}_{\frak p} \colon \Mod_R \to \Mod_{R_{\frak p}}$ as the composite
\[
\Mod_{R} \xr{\matlis_{\frak m}} \Mod^{\text{op}}_{R} \xr{L_{\frak p}} \Mod^{\text{op}}_{R_{\frak p}} \xr{\matlis_{\frak p}} \Mod_{R_{\fp}.}
\]
\begin{lem}\label{lem:dualloc}
For any $M \in \Mod_R$ there is an isomorphism $\cal{L}_{\frak p}\pi_*M \cong \pi_*(\mathbb{L}_{\frak p}M)$.

\end{lem}
\begin{proof}
This follows from commutativity of the diagram 
	\[
\xymatrix@C=35pt{
\Mod_{R} \ar[r]^-{\matlis_{{\frak m}}} \ar[d]_{\pi_*}& \ar[d]_{\pi_*} \Mod^{\text{op}}_{R} \ar[r]^-{L_{\frak p}}& \Mod^{\text{op}}_{R_{\frak p}} \ar[d]_{\pi_*} \ar[r]^-{\matlis_{{\frak p}}} & \Mod_{R_{\fp}} \ar[d]_{\pi_*}\\
\Mod_{\pi_*R} \ar[r]^-{D_{\frak m}} & \Mod^{\text{op}}_{\pi_*R} \ar[r]^-{(-)_{\frak p}}& \Mod^{\text{op}}_{(\pi_*R)_{\frak p}} \ar[r]^-{D_{\frak p}} & \Mod_{(\pi_*R)_{\frak p}}
}
	\]
	which is an easy consequence of \eqref{eq:liftdef} and \Cref{prop:hpsplocal}. 
\end{proof}

\begin{prop}
	Suppose that $\pi_*R$ is such that $\Gamma_{\frak m}R \simeq \Sigma^{\nu}T_R(I_{\frak m})$. Then for each $\fp \in \Spec \pi_*R$ there is a strongly convergent spectral sequence
	\begin{equation}\label{eq:ss2}
(H_{\frak p}^{s}(\pi_*R)_{\frak p})_{t+s} \implies (I_{\frak p})_{t-\nu-d} . 
\end{equation}
	\end{prop}
	\begin{proof}
Since $\pi_*R$ is local we have $\Gamma_{\frak m} \simeq \Gamma_{\cV(\frak m)}$. Recall the spectral sequence of \Cref{prop:homotopy_ss}
\[
(H_{\frak m}^{s}\pi_*R)_{t+s} \implies \pi_{t}(\Gamma_{\cV(\frak m)}R).
\]
Applying $\cal{L}_{\frak p}$ to this spectral sequence and using \Cref{lem:dualloc} and \eqref{eq:matlisdual} we get a spectral sequence 
\[
(H_{\frak p}^{s-d}((\pi_*R)_{\frak p}))_t \implies \pi_{t}(\mathbb{L}_{\frak p}\Gamma_{\cV(\frak m)}R). 
\]
Note that $\Gamma_{\cV(\frak m)}R \simeq \Sigma^{\nu}T_R(I_{\frak m}) \simeq \Sigma^{\nu}\mathbb{D}_{\frak m}R$. Then we have
\[
\pi_{t}(\mathbb{L}_{\frak p}\Gamma_{\cV(\frak m)}R) \cong \pi_{t-\nu}(\mathbb{L}_{\fp}\mathbb D_{\frak m}R) \cong \pi_{t-\nu}(\mathbb{D}_{\frak p}L_{\fp}(\mathbb{D}^2_{\frak m}R))  \cong \pi_{t-\nu}(\mathbb{D}_{\frak p}(R_{\fp})), 
\]
which is isomorphic to $\pi_{t-\nu}(T_{R_{\fp}}(I_{\fp})) \cong (I_{\fp})_{t-\nu}$. Reindexing we see that the spectral sequence is as claimed. 
 \end{proof}
\begin{rem}
If this spectral sequence is isomorphic to that of \Cref{prop:homotopy_ss_local}(1), then using \Cref{lem:homotopycheck} there would be an equivalence $\Gamma_{\frak p}R \simeq \Sigma^{\nu+d} T_R(I_{\frak p})$.
\end{rem} 

\begin{rem}(Dwyer--Greenlees--Iyengar duality)\label{rem:bg}
Suppose that $k$ is a field and $R$ is a coconnective commutative augmented $k$-algebra, with $\pi_*R$ Noetherian, and an isomorphism $\pi_0R \cong k$. Assume additionally that $R \to k$ is Gorenstein of shift $a$ in the sense of \cite[Sec.~8]{dgi_duality}.\footnote{Here we identify $k$ with $Hk$.} Note that by \cite[Rem.~17.1]{greenlees_hi} an algebraically Gorenstein ring spectrum gives rise to a Gorenstein ring spectrum, but the converse need not hold, as the example $R = C^*(BG,k)$ for $G$ a compact Lie group shows. Let $\epsilon = \Hom_R(k,k)$, and assume that $k$ has a unique $\epsilon$-lift \cite[Def.~6.6]{dgi_duality}. Then \cite[Proposition 9.4]{dgi_duality} gives a spectral sequence
\[
(H_{\frak m}^{-s}(\pi_*R))_{t+s} \implies \Hom_k(\pi_{t-a}R	,k) \cong (I_{\frak m})_{t-a}.
\]
Applying Greenlees--Lyubenzik dual localization we get a localized spectral sequence 
\begin{equation}\label{eq:digss}
(H_{\frak p}^{-s}(\pi_*R)_{\frak p})_{t+s} \implies (I_{\frak p})_{t-d-a}.
\end{equation}
Once again, if this spectral sequence were isomorphic to that of \Cref{prop:homotopy_ss_local}(1), then there would be an equivalence $\Gamma_{\frak p}R \simeq \Sigma^{a+d} T_R(I_{\frak p})$.
For example, $R = C^*(BG,k) \to k$ for $G$ a compact Lie group of dimension $w$ is Gorenstein of shift $w$. In this case, Benson and Greenlees show \cite[Thm.~12.1]{bg_localduality} (at least for Lie groups satisfying a mild orientability condition) that $\Gamma_{\frak p}R \simeq \Sigma^{w-d}T_R(I_{\frak p})$. We will give an alternative proof in \Cref{prop:bcliegroup}, and show that $\Gamma_{\frak p}R \simeq \Sigma^{w+d}T_R(I_{\frak p})$. That the sign on $d$ is wrong has also been noted in \cite[Sec.~5]{bikp_modular}. 
\end{rem}
\subsection{The local cohomology of the relative dualizing module}
Suppose we have a morphism $f \colon R \to S$ of ring spectra. As noted above, there always exists a triple $f_* \dashv f^* \dashv f_!$ of adjoint functors. Inspired by the dualizing complexes of algebraic geometry \cite{hartshorne_resduality,neeman_brown,lipman_gd}, Balmer, Dell'Ambrogio, and Sanders \cite{bal_gn_dual} made the following definition. 
\begin{defn}
	The object $\omega_f = f_!(R) = \Hom_R(S,R)$ is called the relative dualizing module.
\end{defn}
The purpose of this section is to calculate $\pi_*\Gamma_{\fp}\omega_f$, the local cohomology of the relative dualizing module, and relate this to work of Benson and Benson--Greenlees \cite{benson_shortproof,bg_localduality}. We start with a lemma which gives conditions for Grothendieck--Neeman duality to hold for $f$. 
\begin{lem}\label{lem:dualizingcomplex}
Suppose that $f\colon R \to S$ is a morphism of ring spectra such that $S$ is a dualizable $R$-module, then $f_*(-) \otimes \omega_f \cong f_!(-)$. Moreover if $f_!$ preserves compact objects, then $\omega_f$ is an invertible $S$-module. 
\end{lem}
\begin{proof}
Since $S$ is a compact (and hence dualizable) $R$-module, $f_!$ preserves colimits, and so by \cite{neeman_brown} or \cite[Thm.~3.3]{bal_gn_dual} there is a natural isomorphism $f_!(-) \cong \omega_f \otimes f_*(-)$. If $f_!$ preserves compact objects then by \cite[Thm.~1.9]{bal_gn_dual} $\omega_f$ is an invertible $S$-module.\footnote{Note the unfortunate clash of notation; in the notation of \cite{bal_gn_dual} $f^\ast$ is the induction functor.}
\end{proof}

Based on the discussion above we make the following definition. 
\begin{defn}\label{def:bc}
	Let $R$ be a ring spectrum. We say that $R$ is \Gorenstein with shift $\nu$ if, for each $\frak p \in \Spec \pi_*R$ of dimension $d$, there is an equivalence $\Gamma_{\frak p}R \simeq \Sigma^{\nu+d} T_R(I_{\frak p})$. More generally, we say that $R$ is \Gorenstein with twist $J$ if there exists an $R$-module $J$ such that there is an equivalence 
\[
\Gamma_{\frak p}R \otimes_R J \simeq \Sigma^{d} T_R(I_{\frak p})
\]
for any $\frak p \in \Spec \pi_*R$ of dimension $d$.
\end{defn}
\begin{rem}
	We use the terminology \Gorenstein as to not conflict with Greenlees' notion of Gorenstein ring spectra \cite{greenlees_hi}. The definition given by Greenlees requires a ring spectrum $R$  equipped with a map $R \to k$. This is said to be Gorenstein of shift $\nu$ if there is an equivalence $\Hom_R(k,R) \simeq \Sigma^\nu k$ of $R$-modules. In the case that $R$ is a $k$-algebra, see \cite[Sec.~18]{greenlees_hi} for conditions that ensure that Gorenstein $k$-algebras are \nsGorenstein, while the general case is covered in \cite{dgi_duality} (both of these only consider the maximal ideal, however it is possible that the methods of dual localization allow this to also be done at an arbitrary prime ideal). 
\end{rem}
For example, if $R$ is algebraically Gorenstein with shift $\nu$, then $R$ is \Gorenstein with shift $\nu$. To find examples, we start with the following definition. 
\begin{defn}\label{def:nn}
We say that $S$ has a \Normalization (of shift $\nu$) if there exists a ring spectrum $R$ and a map of ring spectra $f \colon R \to S$ such that 
	\begin{enumerate}
		\item $R$ is \Gorenstein with shift $\nu$, 
		\item $S$ is a compact $R$-module, and
		\item $\omega_f$ is an invertible $S$-module. 
	\end{enumerate} 
\end{defn}

\begin{rem}
It is possible for $S$ to satisfy the first two conditions, but not the third. As an example, one can take $Hk \to Hk[t,s]/(s^2,t^2,st)$, see \cite[Ex.~4.11]{bal_gn_dual}.
\end{rem}
\begin{rem}
	This definition is inspired by Noether's normalization lemma in commutative algebra. Recall that this says that if $S$ is a finitely-generated graded commutative $k$-algebra with $S_0=k$ and $S_i = 0$ for $i < 0$, then $S$ has a polynomial subring $R = k[x_1,\ldots,x_r]$ over which it is finitely-generated as a module. Indeed, in this case $HR$ is algebraically Gorenstein, and $HS$ is a compact $HR$-module (apply \cite[Lemma.~10.2(i)]{greenlees_hi}), although $\omega_f$ need not be invertible.  
\end{rem}
Before calculating $\pi_*\Gamma_{\fp}\omega_f$, we need a preliminary lemma. 
\begin{lem}[Orthogonality relation]\label{lem:orthogonal}
	Suppose that $\frak p \ne \fp'$ are two prime ideals of $\Spec \pi_*R$. If $M \in \Mod_{R_{\frak p}}^{\frak p-\text{tors}}$ and $N \in \Mod_{R_{\frak p'}}^{\frak p'-\text{tors}}$, then $M \otimes N \simeq 0$.
\end{lem}
\begin{proof}
	We have $\kos{R_{\frak p}}{\frak p} \otimes \kos{R_{\frak p'}}{\frak p'} \simeq 0$ by \cite[Prop.~6.1.7]{hps_axiomatic}. It follows that $M \otimes \kos{R_{\frak p'}}{\frak p'} \simeq 0$ for all $M \in \Loc(\kos{R_{\frak p}}{\frak p})$ and in turn that $M \otimes N \simeq 0$ for all $N \in \Loc(\kos{R_{\frak p'}}{\frak p'})$. By \Cref{lem:genplocptors} this gives the claimed orthogonality result. 
\end{proof}
\begin{thm}\label{thm:bc}
	Suppose $S$ has a \Normalization $f \colon R \to S$ of shift $\nu$. Then for each $\fp \in \Spec \pi_*S$ of dimension $d$, there is an equivalence	
\[
\Sigma^{-d} \Gamma_{\fp}S \otimes \omega_f \simeq \Sigma^{\nu} T_S(I_{\fp}).
\]	
 In particular, we obtain an isomorphism $\pi_\ast(\Gamma_{\fp}\omega_f)  \simeq (I_{\fp})_{\ast-d-\nu}$.
\end{thm}
\begin{proof}
Let us write $\cal{V}_{S,R} \colon \Spec \pi_*S \to \Spec \pi_*R$ for the induced map on $\Spec$. 
Let $\frak q \in \Spec \pi_*R$, and $\cal{U} = (\cal{V}_{S,R})^{-1}(\frak q)$. Note that our assumptions imply that $\cal{U}$ is a finite discrete set. Then we have decompositions
\[
f_*(\Gamma_{\frak q}M) \simeq \bigoplus_{\frak p \in \cal{U}}\Gamma_{\frak p}(f_*M) \quad \text{and} \quad f_!(T_R(I_{\frak q})) \simeq \bigoplus_{\frak p \in \cal{U}} T_S(I_{\frak p}) \label{eq:liftsplit}
\]

	Here the first follows from \cite[Cor.~7.10]{benson_colocalizing_2012} with the functor $F = f_*$, while the second follows from \Cref{prop:liftsplit} and the example before \cite[Prop.~7.1]{bk_pureinjectives}. Moreover, since $R$ is \Gorenstein we have $\Sigma^{\nu+d}T_R(I_{\frak q}) \simeq \Gamma_{\frak q}R$. 

Noting that $f_*R = S$, the decomposition \eqref{eq:liftsplit} and \Cref{lem:dualizingcomplex} give
\begin{equation}\label{eq:decomposition}
\bigoplus_{\fp \in \cU}\Gamma_{\fp}S \otimes \omega_f \simeq f_*(\Gamma_{\fq}R) \otimes \omega_f \simeq f_!(\Gamma_{\fq}R). 
\end{equation}

Fix $\frak p' \in \cal{U}$. Then using the decompositions \eqref{eq:liftsplit} and \eqref{eq:decomposition} and \Cref{lem:orthogonal} we have equivalences
\begin{align*}
	\Gamma_{\frak p'}S \otimes \omega_f &\simeq \Gamma_{\frak p'}S \otimes \bigoplus_{\frak p \in \cal{U}} \Gamma_{\frak p}S \otimes \omega_f \\
	 & \simeq \Gamma_{\frak p'}S \otimes f_!(\Gamma_{\frak q}R) \\
	& \simeq \Gamma_{\frak p'}S \otimes f_!(\Sigma^{\nu+d} T_R(I_{\frak q})) \ \\
	&\simeq \Sigma^{\nu+d} \Gamma_{\frak p'}S \otimes \bigoplus_{\frak p \in \cal{U}}T_S(I_{\frak p}) \\
	& \simeq  \Sigma^{\nu+d}T_S(I_{\frak p'}).
\end{align*}
where the last step follows since $T_S(I_{\frak p}) \in \Mod_{S_{\frak p}}^{\frak p-\text{tors}}$ for each $\frak p$ by \Cref{lem:lifttorlocal}. Now take homotopy, and apply \Cref{lem:cohom}.
	\end{proof}

	\begin{rem}
		It may be possible to prove \Cref{thm:bc} by using the methods of \cite{bg_localduality}; this should certainly be the case when $\omega_f$ is a shift of $S$. The technique used in the proof of \Cref{thm:bc} is inspired by work of Benson \cite{benson_shortproof}. 
	\end{rem}
	
	\begin{rem}
		This theorem should be compared to the following result. Let $(A,\frak m, k)$ be a Noetherian local ring, and $\omega_A	$ a normalized dualizing complex for $A$.  By \cite[Prop.~6.1]{hartshorne_resduality} there is an equivalence $\Gamma_{\frak m}(\omega_A) \simeq I_{\frak m}$ in $\cal{D}(A)$.
	\end{rem}
	
More generally, \Cref{thm:bc} appears to be the shadow of a residual complex formalism in derived algebraic geometry, complementing the global approach taken in \cite[Ch.~II.6]{sag} or \cite[Ch.~II]{gr_book}, in the following sense: Suppose given a commutative diagram
\[
\xymatrix{k \ar[r]^-{r} \ar[rd]_{s} & R \ar[d]^f \\
& S}
\]
of $\mathbb{E}_{\infty}$-ring spectra. We think of $k$ as the base over which the corresponding map $f\colon \Spec(S) \to \Spec(R)$ of derived affine schemes lives. Following Neeman~\cite{neeman_brown}, the absolute dualizing complexes of $R$ and $S$ should then be defined as $\omega_R = r_!(k)$ and $\omega_S = s_!(k)$, respectively. Heuristically, the absolute dualizing complex is equivalent to a residual complex constructed by gluing together the local dualizing complexes of all points $R \to \kappa$ of $\Spec(R)$, and similarly for $S$. The conclusion of \Cref{thm:bc} establishes this compatibility of the absolute dualizing complex and the local dualizing complexes. Furthermore, there is the following simple relation between the absolute and relative dualizing complexes: 

\begin{cor}
With notation as above, assume additionally that $\omega_f$ is dualizable, then there is an equivalence $\omega_R \otimes_R\omega_f \simeq \omega_S$. 
\end{cor}
\begin{proof}
It is clear that there is a natural equivalence of forgetful functors $r^* \circ f^* \simeq s^*$, which yields a natural equivalence $f_!\circ r_! \simeq s_!$ upon passing to right adjoints. Evaluating this equivalence on $k$ thus gives $f_!(\omega_R) \simeq \omega_S$. Combined with \Cref{lem:dualizingcomplex}, we obtain the desired equivalence $\omega_S \simeq f_!(\omega_R) \simeq f_*(\omega_R) \otimes_S \omega_f$.
\end{proof}

In particular, this corollary provides a transitivity relation for Gorenstein normalization. 

\subsection{Examples of \Gorenstein ring spectra}\label{sec:examples}
Suppose $S$ has a \Normalization $f \colon R \to S$. The above theorem is most useful in the case where $\omega_f$ is known for other reasons, and we shall see that there are examples where this is the case. There are three possibilities for $\omega_f$: either it is trivial, a suspension of $S$, or an arbitrary invertible $S$-module. We illustrate each of these cases below. 

To this end, we will work with the coconnective $\mathbb{E}_{\infty}$-ring spectrum $C^*(BG,k)$ of cochains on a compact Lie group $G$ with coefficients in a field $k$. Their basic properties are neatly summarized in the paper \cite{bg_compact_lie} by Benson--Greenlees. In particular, the category $\Mod_{C^*(BG,k)}$ of module spectra over $C^*(BG,k)$ has a symmetric monoidal structure which we will denote by $\otimes_{C^*(BG,k)}$ throughout this section.

Our first example is when $f_*$ and $f_!$ agree up to a shift. The example in this result is due to Benson and Greenlees \cite[Thm.~12.1]{bg_localduality}; we provide a short alternative proof. 
\begin{prop}\label{prop:bccladjoint}
	Suppose $S$ has a \Normalization with shift $\nu$ such that $f_* \simeq \Sigma^{\ell} f_!$. Then $S$ is \Gorenstein with shift $\nu + \ell$. For example, this is the case if $S = C^*(BG,k)$ for $G$ a compact Lie group of dimension $w$, where the adjoint representation of $G$ is orientable over $k$, or $k$ has characteristic 2. In this case, $C^*(BG,k)$ is \Gorenstein with shift $w$.
\end{prop}
\begin{proof}
	That $S$ is \Gorenstein with shift $\nu + \ell$ is just a special case of \Cref{thm:bc}. For the example, let us write $C^*(BG)$ for $C^*(BG,k)$ for brevity. There is always a faithful representation $ G \to SU(n)$, which gives rise to a morphism $f \colon C^*(BSU(n)) \to C^*(BG)$.  Recall that $H^*(SU(n),k) \cong k[c_2,\ldots,c_n]$ is a regular local ring (in particular, it is Gorenstein). It follows that $C^*(SU(n))$ is Gorenstein of shift $n^2-1$, and so we can apply \Cref{prop:grosshopkinsgorenstein} to see that for $\frak q$ of degree $d$ we have $\Gamma_{\frak q}C^*(BSU(n)) \simeq \Sigma^{n^2-1+d}T_{C^*(BSU(n))}(I_{\frak q})$ (see also \cite[Thm.~10.6(i)]{bg_localduality}, but note that the sign there is incorrect).
		
		A theorem of Venkov \cite{venkov} implies that $H^*(G,k)$ is finitely-generated as a module over $H^*(SU(n),k)$. Combining \cite[Lem.~10.2(i)]{greenlees_hi} and \cite[Thm.~2.1.3]{hps_axiomatic} we see that $C^*(BG)$ is a compact $C^*(BSU(n))$-module. Furthermore, by \cite[Thm.~6.10]{bg_compact_lie} and the remark following, we have $f_! \simeq \Sigma^{w-n^2+1}f_*$ so that $C^*(BG)$ has a \Normalization of shift $n^2-1$ and $C^*(BG)$ is \Gorenstein with shift $w$. 
\end{proof}
In particular, if $G$ is a finite group $C^*(BG)$ is absolute Gorenstein of shift $0$.
\begin{rem}
	In \cite{bg_localduality} it was claimed that $\Sigma^{d}\Gamma_{\fp}C^*(BG,k) \simeq \Sigma^{w}T_{C^*(BG)}I_{\fp}$, however, as noted in \cite[Sec.~5]{bikp_modular}, the sign on $d$ is incorrect. 
\end{rem}
We note that we can extend this example to the case of a general compact Lie group. 
\begin{prop}\label{prop:bcliegroup}
Let $G$ be a compact Lie group of dimension $w$ where the adjoint representation is not orientable, and $k$ does not have characteristic 2. Then there exists a sign representation $\epsilon$ of $G$ such that for each $\fp \in \Spec H^*(G,k)$ of dimension $d$ there is an equivalence
\[
\Sigma^{-d} \Gamma_{\fp}C^*(BG,k) \otimes_{C^*(BG,k)} C^*(BG,\epsilon) \simeq \Sigma^{w}T_{C^*(BG)}I_{\fp}, 
\]
i.e., $C^*(BG,k)$ is \Gorenstein with invertible twist. 
\end{prop}
	\begin{proof}
This is similar to the previous argument, however in this case the dualizing module is not simply a suspension. In fact there exists an index 2 subgroup of $G$ (see the bottom of page 42 of \cite{benson_greenlees_lie}), and a sign representation $\epsilon$ such that $f_! \simeq \Sigma^{w-n^2+1}C^*(BG,\epsilon) \otimes_{C^*(BG,k)} f_*$ (see \cite[Rem.~6.11]{bg_compact_lie}), and we have $
\Sigma^{-d} \Gamma_{\fp}C^*(BG,k) \otimes_{C^*(BG,k)} C^*(BG,\epsilon) \simeq \Sigma^{w}T_{C^*(BG)}I_{\fp}$,	as required. Finally, the $C^*(BG,k)$-module $C^*(BG,\epsilon)$ is invertible by \cite[Cor.~6.9]{bg_compact_lie}, since $SU(n)$ is connected.
	\end{proof}

\section{Local duality for finite-dimensional Hopf algebras}\label{sec:hopfalgebras}
The purpose of this section, which can be considered as an extended example based on the previous sections, is to investigate local duality for the cellular objects in the category of comodules over a finite-dimensional Hopf algebra $B$ over a field $k$ (equivalently, quasi-coherent sheaves on the finite group scheme associated to $B$). 

The abelian category of comodules over a finite-dimensional Hopf algebra over $k$ is a locally Noetherian Grothendieck abelian category with a closed symmetric monoidal structure; this is a special case of the study of Hopf algebroids in \cite[Sec.~4]{bhv} (that $\Comod_B$ is locally Noetherian appears in the proof of Lem.~4.17 of \emph{loc. cit.}). We also refer the reader to \cite[Sec.~9.5]{hps_axiomatic} for direct proofs of some standard properties of $\Comod_B$. 

Following \cite[Sec.~4]{bhv} we define a stable category $\Stable_B$ as a replacement for the usual derived category of comodules. This turns out to be an $\infty$-categorical version of the (triangulated) category $\cal{C}(B)$ studied in \cite[Sec.~9.6]{hps_axiomatic}, however since we do not need this result, we do not prove it. 
\begin{defn}\label{defn:stable_comod}
Let $\cG_d$ be a set of representatives of dualizable $B$-comodules. We define the stable $\infty$-category of $B$-comodules as $\Stable_{B} = \Ind(\Thick_{B}(\cG_d))$.
\end{defn}
This is a stable category compactly generated by the objects $\cG_d$. It is a simple check with the Jordan--H{\"o}lder theorem to see that an object is dualizable in $\Comod_B$ if and only if it is a simple comodule, so that these also provide a set of compact generators of $\Comod_B$. The endomorphism ring is given by $\pi_{-\ast}\Hom_{\Stable_B}(k,k) \cong \Ext^\ast_B(k,k)$; this follows, for example, by \cite[Cor.~4.19]{bhv}. Friedlander and Suslin~\cite{fs_97} have proved that $\Ext^\ast_B(k,k)$ is a Noetherian graded local ring. 
In light of this, we make the following definition. 
\begin{defn}
	For $M \in \Stable_{B}$ we define $\Ext^\ast_B(k,N) = \pi_{-\ast}\Hom_{\Stable_B}(k,N)$. For simplicity, we will usually just write $H^*(B,N)$. If $N$ is $m$-coconnective, then by \cite[Cor.~4.19]{bhv}, we see this agrees with the classical definition of $\Ext$ in the category of $B$-comodules.
\end{defn}

Let $\Stablec:=\Loc(k)$ be the full subcategory of cellular objects in $\Stable_B$. Given any specialization closed subset $\mathcal V$ of $\Spec H^*(B,k)$, we consider the full subcategory of $\Stablec$
\[
\Stablec^{\mathcal{V}-\text{tors}}=\Loc(\kos{k}{\frak p} \mid \frak p \in \mathcal{V}),
\]
where the Koszul object $\kos{k}{\frak{p}}$ is constructed as in \Cref{sec:ringspectra}. Since $k$ is clearly compact in $\Stablec$, we can use local duality to obtain local homology and cohomology functors.
\begin{thm}\label{thm:cellular_functors}
	There is a quadruple of functors $(\Gamma_{\mathcal{V}}, L_{\mathcal{V}}, \Delta^{\mathcal{V}}, \Lambda^{\mathcal{V}})$ on $\Stablec$ satisfying all the properties in \cref{thm:abstractlocalduality}. In addition, $\Gamma_{\mathcal{V}}$ and $L_{\mathcal{V}}$ are smashing.
\end{thm}
\begin{proof}
	Consider the local duality context $(\Stablec,\Stablec^{\mathcal{V}-\text{tors}})$. The first claim follows from \cref{thm:abstractlocalduality} while the second one is a consequence of \cref{lem:smashing}.
\end{proof}

\begin{rem}
	The functors defined above can be extended to $\Stable_B$, resulting in a quadruple naturally equivalent to the one yielded by the local duality context $(\Stable_B, \Locid{}(\kos{k}{\frak p} \mid \frak p \in \mathcal{V}))$. This is completely formal, but we omit here for the sake of conciseness.
\end{rem}

Now, according to \cref{prop:moritacell}, there is a symmetric monoidal equivalence of stable categories
\begin{equation}\label{eq:morita_equiv}
\xymatrix{\Phi: \Stablec \ar[r]^-{\sim} & \Mod_{R},}
\end{equation}
where $\Phi=\Hom_{\Stable_B}(k,-)$ and $R=\End_{\Stable_B}(k)$. As we shall see, the equivalence in (\ref{eq:morita_equiv}) allows us to easily transfer results from ring spectra to $\Stable_B$ that may have been difficult to prove directly. Before we continue, it is worth pointing out that in some cases we can give a more concrete description of the cellular objects. 

\begin{prop}
If $B = (kG)^*$ for a finite group $G$, then there is an equivalence
\[
\Loc(k) \simeq  \Mod_{C^*(BG,k)}
\]
where $C^*(BG,k)$ denotes the $\mathbb{E}_{\infty}$-ring spectrum of $k$-valued cochains on the classifying space $BG$ of $G$. In the case that $G$ is a finite $p$-group we have 
\[
\Stable_{(kG)^*} \simeq \Mod_{C^*(BG,k).} 
\]
\end{prop}
\begin{proof}
Again letting $R = \End_{\Stable_B}(k)$ we have that $\Loc(k) \simeq \Mod_R$. By the Rothenberg--Steenrod construction we have $R \simeq C^*(BG,k)$, see \cite[Sec.~4]{krausebenson_kg} for example. In the case where $G$ is a finite $p$-group, $\Stable_{(kG)^*}$ is monogenic, so that it is isomorphic to $\Loc(k)$.  
\end{proof}

\begin{rem}
More generally, there is a version of this proposition for compact Lie groups $G$, where the $\infty$-category $\Fun(BG,\Mod_k)$ of $k$-valued local systems on $BG$ plays the role of $\Stable_{kG}$. 
\end{rem}

\begin{prop}\label{prop:Mod_Stable_equiv}
 	Given a specialization closed subset $\mathcal{V}\subseteq \Spec H^*(B,k)$, the functor $\Phi$ restricts to an equivalence of full subcategories
 	\[
 	\xymatrix{\Phi: \Loc_{\Stable_B}(\kos{k}{\frak p}\mid \frak p \in \mathcal{V}) \ar[r]^-{\sim} & \Loc_{\Mod_R}(\kos{R}{\frak p}\mid \frak p \in \mathcal{V}).}
 	\]
\end{prop} 
\begin{proof}
	It suffices to show that $\Phi(\kos{k}{\frak p})\simeq \kos{R}{\frak p}$. Let $a$ be a homogeneous element of degree $d$ in $H^*(B,k)$. Then,
	\[
	\Phi(k \xr{a} \Sigma^{d} k \to \kos{k}{a})=(R \xr{\Phi(a)} \Sigma^{d} R \to \Phi(\kos{R}{a})).
	\]
	Under the isomorphism $H^s(B,k)\cong \pi_{-s}R$, we can identify $\Phi(a)$ with $a$ in degree $-d$. Hence, we get $\Phi(\kos{k}{a})\simeq \kos{R}{a}$. The conclusion follows from the observation that $\kos{k}{\fp} \simeq \kos{k}{a_1} \otimes \cdots \otimes \kos{k}{a_n}$ for $\fp = (a_1,\ldots,a_n)$, together with the fact that $\Phi$ is a symmetric monoidal equivalence.
\end{proof}

When $\cal{V}=\cal{V}(\p)$, we get a more explicit description of the local homology and cohomology functors.

\begin{prop}
	Let $\frak p$ in $\Spec H^*(B,k)$. The local duality contexts
	$(\Stable_B,\Loc(\kos{k}{\frak p}))$  and  $(\Stable_B,\Loc(\kos{k}{\frak q} \mid \frak q \in \cal{V}(\frak p)))$
	are equivalent.
\end{prop}
\begin{proof}
	This is immediate in light of \cref{prop:Mod_Stable_equiv} and \cref{prop:loccomparasion}.
\end{proof}

Likewise, \cref{prop:Mod_Stable_equiv} together with \cref{cor:formulas} yield: 
\begin{prop}
	Let $\frak p=(p_1, p_2, \dots p_n)$ in $\Spec H^*(B,k)$. For every $M\in \Stablec$, there are natural equivalences 
	\[
\Gamma_{\cV(\frak p)}M \simeq \colim_s \Sigma^{-sd-n} \kos{k}{\frak p^s} \otimes M \quad \text{and} \quad \Lambda^{\cV(\frak p)}M \simeq \lim_s  \kos{k}{\frak p^s} \otimes M,
	\]
where $d = |p_1|+\cdots + |p_n|$ and the objects $\kos{k}{\frak p^s}$, $s\geq1$, are defined as in \Cref{sec:ringspectra}.
\end{prop}

Another consequence of \eqref{eq:morita_equiv} is the existence of spectral sequences computing the cohomology of the different local cohomology and local homology functors, which converge for cellular objects in $\Stable_B$. 
\begin{prop}\label{prop:homotopy_ss_comod}
	Let $\frak p$ in $\Spec H^*(B,k)$ and $M\in \Stablec$. There are strongly convergent spectral sequences of $H^*(B,k)$-modules:
	\begin{enumerate}
		\item $E^2_{s,t}=(H^{s}_{\frak p}(H^*(B,M)))_{t} \implies H^{s+t}(\Gamma_{\cV(\frak p)} M)$, with differentials $d^r\colon E^r_{s,t} \to E^r_{s-r,t+r-1}$,
		\item $E^2_{s,t}=(H_{-s}^{\frak p}(H^*(B,M)))_{t} \implies H^{s+t}(\Lambda^{\cV(\frak p)} M)$, with differentials $d^r\colon E^r_{s,t} \to E^r_{s-r,t+r-1}$, and
		\item $E^2_{s,t}=(\cH^{s}_{\frak p}(H^*(B,M)))_{t} \implies H^{s+t}(L_{\cV(\frak p)} M)$, with differentials $d^r\colon E^r_{s,t} \to E^r_{s-r,t+r-1}$.
	\end{enumerate}
\end{prop}
\begin{proof}
	The result follows from applying the equivalence \eqref{eq:morita_equiv} to the spectral sequences in \cref{prop:homotopy_ss}.
\end{proof}

\begin{rem}
	We note that constructing these spectral sequences directly via a filtration of the Koszul complex is possible, though determining convergence from this perspective is difficult. One advantage of our approach is that the convergence of the spectral sequence is immediate. 
\end{rem}

Next, observe that the equivalence in (\ref{eq:morita_equiv}) gives the analogous result to \Cref{prop:hpsplocal}.
\begin{lem}
For any $\frak p \in \Spec H^*(B,k)$ there exists a smashing localization $L_{\frak p}$ such that for all $M$ in $\Stablec$ there is an isomorphism $H^*(B,L_{\fp}M) \cong H^*(B,M)_{\frak p}$.
\end{lem}
Let us write $\Stablec_\p$ for the essential image of $L_{\fp}$. Given $M\in\Stablec$, we will write $M_{\frak p}$ for $L_{\frak p}M$ to ease notation. 
\begin{lem}\label{lem:plocmorita}
The category $\Stablec_{\fp}$ is equivalent to $\Mod_{R_{\fp}}$.
\end{lem}
\begin{proof}
This is an application of derived Morita theory. First, it is easy to see that $\Stablec_{\fp} \simeq \Loc(k_{\fp})$. Now note that $\pi_*\Hom_{\Stablec_\fp}(k_{\fp},k_{\fp}) \cong \pi_*\Hom_{\Stablec}(k,k_{\fp}) \cong (\pi_*R)_{\fp}\cong \pi_*R_\p$ and the result follows. 
\end{proof}

In analogy with ring spectra, we introduce the definitions below where once again we write $\Delta^{\fp}$ for the right adjoint to $L_{\fp}$.  
\begin{defn}\label{def:pcohomologystable}
	For $\frak p \in \Spec H^*(B,k)$ and $M \in \Stablec$ define the $\frak p$-local cohomology and homology functors by $\Gamma_{\frak p}M = \Gamma_{\cV(\frak p)}M_{\frak p}$ and $\Lambda^{\frak p}M = \Lambda^{\cV(\frak p)}\Delta^{\frak p}M$, respectively. 
\end{defn}
We let $\Stablec_\fp^{\fp-\text{tors}}$ denote the essential image of $\Gamma_{\fp}$. Combining \Cref{prop:Mod_Stable_equiv} and \Cref{lem:plocmorita} we get the following. 
\begin{prop}\label{prop:gammap}
The functor $\Phi \colon \Stablec \xr{\simeq} \Mod_R$ restricts to an equivalence $\Stablec_{\fp}^{\fp-\text{tors}} \xr{\simeq} \Mod_{R_{\fp}}^{\fp-\text{tors}}$.
\end{prop}  

\Cref{prop:homotopy_ss_local} translates into the spectral sequence below. 
\begin{prop}\label{prop:localcohomssnew}
Let $\fp \in \Spec H^*(B,k)$ be a homogeneous ideal. For any $M \in \Stablec$ there exists a strongly convergent spectral sequences of $H^*(B,k)$-modules 
	\[
	E^2_{s,t}\cong(H^{s}_{\frak p}H^*(B,M)_{\frak p})_{t} \implies H^{s+t}(B,\Gamma_{\frak p} M),\] with differentials $d^r\colon E^r_{s,t} \to E^r_{s-r,t+r-1}.$
\end{prop}
\begin{rem}
This spectral sequence is a generalization of a spectral sequence due to Greenlees--Lyubeznik \cite{green_lyub} and Benson \cite{benson_moduleswithinjcohom} when $B = (kG)^*$ for $G$ a finite group and $M = k$.
\end{rem}
\tabularnewline
\begin{rem}
	If we extend the definition of $\Gamma_{\fp}$ to all $M \in \Stable_B$ via $\Gamma_{\fp}M = \Gamma_{\fp} k \otimes M$, then we can make sense of this spectral sequence for any $M \in \Stable_B$. First note that since $\Cell$ is right adjoint to a symmetric monoidal functor, it is lax symmetric monoidal. This implies that for any cellular object $N$ and arbitrary $M$, there is a natural map
	\[
N \otimes \Cell(M) \to \Cell(N \otimes M).
	\]
The collection of $N$ for which this is an equivalence is a localizing subcategory containing $k$, and so contains all cellular objects. Since $\Gamma_{\fp} k$ is always cellular, we see that 
\[
\Cell(\Gamma_{\fp}M) \simeq \Cell(\Gamma_{\fp}k \otimes M) \simeq \Gamma_{\fp}k \otimes \Cell(M) \simeq \Gamma_{\fp}\Cell(M).
\] 
But for any $M \in \Stable_B$ there is an isomorphism $H^*(B,\Cell(M)) \cong H^*(B,M)$ (see the proof of \Cref{lem:bousfield} for example), so that the equivalence above gives the spectral sequence for all objects of $\Stable_B$. We thank the referee for this observation. 
\end{rem}

Finally, we proceed to compare our functors with the ones constructed in \cite{benson_local_cohom_2008} and \cite{benson_colocalizing_2012}. Let $Z\in \Stable_B^{\omega}$ be a compact object. For all $X\in \Stable_B$ set $H^*_Z(B,X)= \pi_*\Hom_{\Stable_B}(Z,X)$ to be the cohomology of $X$ with respect to $Z$.

 For any specialization closed subset $\cV \subseteq \Spec H^*(B,k)$, Benson, Iyengar, and Krause construct a localization functor on $\Stable_B$ (their parallel to our $L_{\cV}$) whose kernel is precisely 
	\[
T_{\cV} = \{X \in \Stable_B \mid H_Z^*(B,X)_{\frak p} = 0 \text{ for all } \frak p \in \Spec H^*(B,k) \setminus \cV, \ Z\in \Stable_B^{\omega} \}.
	\]
We have the following comparison result. 
\begin{prop}
	For any specialization closed set $\mathcal{V}$, there is an equivalence of categories between $\Loc(k)\cap T_{\cV}$ and $\Loc(\kos{k}{\frak p} \mid \frak p \in \cal{V})$.
\end{prop}
\begin{proof}
Let $X$ be a cellular object. Note that if $H^*(B,X)_{\frak p}\cong H^*(B,X_{\frak p})=0$, then $X_{\frak p}\simeq 0$. Therefore, $H^*(B,X)_{\frak p}=0$ implies that $H^*_Z(B,X)_{\frak p}=0$ for all $Z\in \Stable_B^{\omega}$. The statement then follows from \eqref{eq:morita_equiv} and \cref{prop:bikcomparespecial}.
\end{proof}

We finish by observing that the same proof of \cref{cor:bikcompare} yields a comparison of these functors to those constructed by Benson, Iyengar, and Krause. 
\begin{cor}
	The functors $\Gamma_{\frak p}$ and $\Lambda^{\frak p}$ agree with the restriction of the equally denoted functors constructed by Benson, Iyengar, and Krause in \cite{benson_local_cohom_2008} and \cite{benson_colocalizing_2012} to the subcategory of cellular objects. 
\end{cor}
\subsection{The \Gorenstein condition for $\Stable_B$}
The purpose of this short section is to show that if the subcategory of cellular objects in $\Stable_B$ satisfies the \Gorenstein condition, then, in a sense, so does $\Stable_B$ (and in fact this holds in any stable category). For clarity, let us write $R = \Hom_{\Stable_B}(k,k)$, so that $\pi_{-*}R \cong H^*(B,k)$, and the subcategory of cellular objects in $\Stable_B$ is equivalent to $\Mod_R$. 

We first construct the analogue of the objects $T_R(I)$ of \Cref{sec:gorenstein}. Using Brown representability, for each injective $H^*(B,k)$-module $I$ there is an object $T_B(I)$ with the property that for any $M \in \Stable_B$
\[
\pi_*\Hom_{\Stable_B}(M,T_B(I)) \cong \Hom_{H^*(B,k)}(H^*(B,M),I). 
\]
Once again, by taking $M = k$ we see that $H^*(B,T_B(I)) \cong I$ as modules over $H^*(B,k)$. 

\begin{prop}
	If $R$ is \Gorenstein with shift $\nu$, then, for any $\fp \in \Spec H^*(B,k)$ of dimension $d$, there is an equivalence in $\Stable_B$
	\[
\Cell T_B(I_{\fp}) \simeq \Sigma^{d+v}\Gamma_{\fp}k,
	\]
where $\Cell$ is the right adjoint to the inclusion $\iota:\Stablec\to\Stable_B$ as in \cref{ssec:cell}.
\end{prop}
\begin{proof}
	As previously, we write $\Phi \colon \Stable_B \to \Mod_R$ for the functor $\Hom_{\Stable_B}(k,-)$. Since $\pi_*\Phi(\Cell T_B(I)) \cong \pi_*\Phi(T_B(I)) \cong I$ we can apply \Cref{lem:homotopycheck} to see $\Phi(\Cell T_B(I)) \simeq T_R(I)$. Since $R$ is \Gorenstein of shift $\nu$, for any $\fp \in \Spec H^*(B,k)$ we have $T_R(I_{\fp}) \simeq \Sigma^{d+v}\Gamma_{\fp}R $. Thus, using \Cref{prop:gammap}, we have 
\[
\Phi(\Cell T_B(I)) \simeq T_R(I)\simeq \Sigma^{d+v}\Gamma^R_{\fp}R \simeq \Phi(\Sigma^{d-v}\Gamma_{\fp}k). 
\]
Since everything in sight in cellular, $\Phi$ is an equivalence, and so applying its inverse, we conclude that $\Cell T_B(I_{\fp}) \simeq \Sigma^{d+v}\Gamma_{\fp}k$, as required. 
\end{proof}
In the case that $B = (kG)^*$ for $G$ a finite group the results of \cite{benson_shortproof} imply that $T_{B}(I)$ is already cellular. We suspect this is always the case, however we are unable to prove this directly. Note that it is always in the $\pi$-local category. Thus we leave this as an open question:
\begin{quest}
Is $T_B(I)$ always cellular? 
\end{quest}
When $G$ is a finite $p$-group, $k$ is the only simple $B=(kG)^*$-comodule \cite[Sec. 9.5]{hps_axiomatic}. Hence, $\Stable_{(kG)^*}$ is compactly generated by $k$, so that all objects are cellular. Hence, we get:
\begin{cor}
	If $B = (kG)^*$ for $G$ a finite $p$-group, then for any $\fp \in \Spec H^*(B,k)$ of dimension $d$, there is an equivalence $T_B(I_{\fp}) \simeq \Sigma^{d+v}\Gamma_{\fp}k$ in $\Stable_B$.
\end{cor}

\bibliography{duality}\bibliographystyle{alpha}
\end{document}